\documentclass[reqno]{amsart}

\let\oldtocsection=\tocsection

\let\oldtocsubsection=\tocsubsection

\let\oldtocsubsubsection=\tocsubsubsection

\renewcommand{\tocsection}[2]{\hspace{0em}\oldtocsection{#1}{#2}}
\renewcommand{\tocsubsection}[2]{\hspace{1em}\oldtocsubsection{#1}{#2}}
\renewcommand{\tocsubsubsection}[2]{\hspace{2em}\oldtocsubsubsection{#1}{#2}}

\usepackage{amsmath}
\usepackage{mathtools}
\usepackage{amsfonts}
\usepackage{amssymb}
\usepackage{graphicx}
\usepackage{verbatim}
\usepackage{comment}
\usepackage[colorlinks=true]{hyperref}

\usepackage{multicol, color}

\newcounter{mnote}
\newcounter{mnoteE}
\usepackage{graphicx,psfrag}
\numberwithin{equation}{section}
\begin{document}
\title [Data Assimilation and a determining form for KdV]{Determining form and data assimilation algorithm for weakly damped and driven Korteweg-de Vries equaton- Fourier modes case}
\date{October 8, 2015}
%
\author{Michael S. Jolly}
\address[Michael S. Jolly]{Department of Mathematics\\
                Indiana University, Bloomington\\
        Bloomington , IN 47405, USA}
\email[Michael S. Jolly]{msjolly@indiana.edu}
\author{Tural Sadigov}
\address[Tural Sadigov]{Department of Mathematics and Sciences\\
                SUNY Polytechnic Institute\\
        Utica, NY 13502, USA}
\email[Tural Sadigov]{sadigot@sunyit.edu}
\author{Edriss S. Titi}
\address[Edriss S. Titi]{Department of Mathematics, Texas A\&M University, 3368 TAMU, College Station, TX 77843-3368, USA. Also: The Department of Computer Science and Applied Mathematics\\
The Weizmann Institute of Science, Rehovot 76100, Israel.}
\email[Edriss S. Titi]{titi@math.tamu.edu and edriss.titi@weizmann.ac.il}

\maketitle{}
\begin{abstract}
We show that the global attractor of a weakly damped and driven Korteweg-de Vries equation (KdV) is embedded in the long-time dynamics of an ordinary differential equation called a determining form. In particular, there is a one-to-one identification of the trajectories in the global attractor of the damped and driven KdV and the steady state solutions of the determining form. Moreover, we analyze a data assimilation algorithm (down-scaling) for the weakly damped and driven KdV. We show that given a certain number of low Fourier modes of a reference solution of the KdV equation, the algorithm recovers the full reference solution at an exponential rate in time.   
\end{abstract}


\section{Introduction}

The Korteweg-de Vries equation
	\begin{align}
		u_t+uu_{x}+u_{xxx}=0. \label{kdv}
	\end{align}
was derived as a model of unidirectional propagation of water waves with small amplitude in a channel. It was first introduced by Bousinessq and then reformulated by Diederik Korteweg and Gustav de Vries in 1885. The function $u(x,t)$ in $\eqref{kdv}$ represents the elongation of the wave at time $t$ and position $x$. The solutions of this nonlinear and dispersive equation are solitary waves. In physical applications, however, one can expect some dissipation of energy, as well as external excitation. To account for these effects, damping and forcing terms are added to the model
	\begin{align}
		u_t+uu_{x}+u_{xxx}+\gamma u=f. \label{KdV}
	\end{align}
Existence and the uniqueness of the solution of the damped and driven Korteweg-de Vries (KdV) equation subject to the boundary conditions
	\begin{align}
		u(t, x)=u(t, x+L), \qquad \forall \text{ } (t, x)\in \mathbb{R}\times \mathbb{R}, \label{boundary}
	\end{align}
can be shown by adjusting the methods used for undamped KdV in \cite{BonaSmith}, \cite{Lions} or \cite{Temam1}. The existence of the weak global attractor $\mathcal{A}$, for $\eqref{KdV}$-$\eqref{boundary}$, was shown in \cite{Gh2}, and the strong global attractor  in $H^{2}$  was shown in  \cite{Gh3}.  In particular, it has been shown in \cite{Gh3} that there exist a constant $R$, that depends only on $\gamma$ and $|f|_{H^{2}}$, such that
	\begin{align}
		\sup_{s\in \mathbb{R}} |u(s)|_{H^{2}}\leq R, \label{Rkdv1}
	\end{align}
for every $u(\cdot)\subset \mathcal{A}$. We observe that the estimates detailed in section 4, below, can be followed almost line by line in order to obtain an explicit bound for $R$.

For many strongly dissipative PDE's, capturing the attractor by a finite system of ordinary differential equations is achieved by restricting the equation to an inertial manifold, as is done for Kuramoto-Sivashinsky, Ginzburg-Landau and certain reaction-diffusion equations (see, e.g., \cite{TemamDyn} and references therein). An inertial manifold is a finite dimensional Lipschitz positively invariant manifold which attracts all the solutions at an exponential rate.  A sufficient condition for the existence of an inertial manifolds is the presence of large enough gaps in the spectrum of the linear dissipative operator, i.e.~the presence of separation of scales in the underlying dynamics. The existence of inertial manifolds is still out of reach for various dissipative equations, including the two-dimensional Navier-Stokes equations, and the damped and driven KdV equation \eqref{KdV}-\eqref{boundary}. Our aim here is  to capture the attractor in $H^{2}$, of the damped and driven KdV, by the dynamics of an ordinary differential equation, called  a \emph{determining form}, which is defined in the phase space of trajectories.

A determining form is found for the 2D Navier-Stokes equations (NSE) in \cite{Form1} by using finitely many determining modes. In that work, the trajectories in the attractor of the 2D NSE are identified with traveling wave solutions of the determining form. Another type of  determining form is found for the 2D NSE by the same authors in  \cite{Form2}. The steady state solutions of this second kind of determining form are precisely the trajectories in the global attractor of the 2D NSE. Dissipativity (viscosity) plays a fundamental role in establishing a determining form for this equation.

In contrast, the weakly damped and driven nonlinear Schr\"odinger equation (NLS) and weakly damped and driven KdV are dispersive equations. They are not strongly dissipative due to the absence of viscosity. To embed the attractors of these systems in the long time dynamics of ordinary differential equations requires different techniques. Recently, we have shown that a determining form of the second kind exists for the damped and driven NLS (see \cite{JST}) using a feedback control term involving the Fourier projection of a trajectory in the attractor. In this paper we adapt the approach in \cite{Form2} and \cite{JST} for the KdV.
As in \cite{JST} the analysis here uses compound functionals motivated by the Hamiltonian structure of the corresponding systems.

The idea for determining forms starts with the property of determining modes (see \cite{modes}). A projector $P$ is said to be determining if whenever $u_{1}(\cdot), u_{2}(\cdot) \subset \mathcal{A}$ have the same projection $Pu_{1}(t)=Pu_{2}(t)$ for all $t\in \mathbb{R}$, they are in fact the same solution. A determining projector $P$ defines a map $W$ on the set $\mathcal{S}=\{Pu(\cdot) | u(\cdot)\subset \mathcal{A} \}$. A key step in constructing a determining form is to extend this map to a function space. If $P=P_{N}$ is the projection onto the first $N$ Fourier modes, the number $N$ is called the number of determining modes. Like the dimension of $\mathcal{A}$, $N$ serves as a measure of the complexity of the flow, and the resolution required to capture it. We give an estimate for $N$ in terms of the damping parameter $\gamma$ and the forcing term $f$ for the KdV.

The analysis of the determining form is akin to that for our data assimilation (down-scaling) algorithm. Data assimilation refers to the injection of coarse grain observational data into the model to drive the system toward an accurate solution (see e.g., \cite{Korn}, \cite{Olson_Titi_2003}, \cite{Hayden_Olson_Titi}, \cite{AJT}, \cite{Hakima}, and for computational study of continuous data assimilation, \cite{GOT}). The proof that this works for the KdV is similar to extending the $W$ map except instead of taking an initial time to $-\infty$, we take the time to $+\infty$. Indeed the feedback control approach to determining forms was inspired by that approach to data assimilation (down-scaling) first taken in \cite{AzOT} and \cite{AzT}. As in \cite{JST}, we use a `reverse' Poincar\'e type inequality to extend the $W$ map and show that it is Lipschitz. For the data assimilation algorithm, we use a different technique which involves combining two differential inequalities.

\section{Preliminaries}
Consider the 1D damped and driven KdV  equation
	\begin{align*}
		u_t+uu_{x}+u_{xxx}+\mathcal{L}(u)=f.
	\end{align*}
The damping term $\mathcal{L}(u)$ could have different forms such as $-\nu u_{xx}$ or $\gamma u$. The case where $\mathcal{L}(u)=-\nu u_{xx}$ constitutes stronger dissipation. The analysis for a data assimilation algorithm and a determining form for the case $\mathcal{L}(u)=-\nu u_{xx}$ is very similar to the analysis in \cite{Form2}, and will not be discussed here. In this paper, we will focus on the case where $\mathcal{L}(u)=\gamma u$. Thus, our equation is
	\begin{align*}
		u_t+uu_{x}+u_{xxx}+\gamma u=f,
	\end{align*}
subject to periodic boundary conditions
	\begin{align*}
		u(t, x)=u(t, x+L), \qquad \forall \text{ } (t, x)\in \mathbb{R}\times \mathbb{R},
	\end{align*}
and initial value $u(0, x)=u_0(x)$, where $0<L<\infty$ and $\gamma >0 $. Let $0\leq k< \infty$. We denote by $H^k[0, L]$ (or simply $H^k$) the Sobolev space of order $k$,
$$
H^k[0, L]:=\left \{ u\in L^2[0, L]:    D^\alpha u\in L^2[0, L]\,\, \text{for} \,\, 0\leq \alpha \leq k \right \},
$$
and by $\dot{H}_{\text{\text{per}}}^k$, the subspace of $H^k$ consisting of functions which are periodic in $x$, with period $L$, and have spatial mean zero. Note that $\dot{H}_{\text{per}}^0[0, L]= \dot{L}_{\text{per}}^2[0, L]$. We assume that $u_0\in \dot{H}_{\text{per}}^2$, $f$ is time independent, and $f\in \dot{H}_{\text{per}}^2$. It is easy to see that if $u_{0}$ and $f$ are with spatial mean zero, then the solution of $\eqref{KdV}$ has spatial mean zero. It has been proven in \cite{Gh2} that $\eqref{KdV}$, subject to the boundary conditions $\eqref{boundary}$, has a strong global attractor in $H^{2}$. The global attractor is the maximal compact invariant set under the solution operator $S(t, \cdot)$. Throughout this paper, for simplicity, we omit the domain of integration and $dx$ in the spatial integration so that
	$$|u|^{2}:=\int u^{2}, \qquad |u|_{H^1}^2:=\int u_{x}^2,$$
	$$|u|_{H^2}^2:=\int u_{xx}^2, \qquad |u|_{\infty}:=\sup_{x\in[0,L]}|u(x)|.$$
We also recall the Agmon inequality
	\begin{align}
		\sup_{x\in[0,L]}|w(x)|\leq |w|^{\frac{1}{2}} |w_{x}|^{\frac{1}{2}} \label{Agmon}.
	\end{align}	
Bounding expressions that depend on $\gamma$, $f$ (and $\mu$, see $\eqref{KdVw}$, below) will be denoted by capital letters $\mathcal{R}$ with specific indices. The bounding expressions $\mathcal{R}$ with indices $0$, $1$, $\infty$, $2$ and superscript $'$ are the $L^2$, $H^1$, $L^{\infty}$, $H^2$ and time derivative bounds, respectively,  for the solution in the global attractor of $\eqref{KdVw}$, below. Those bounding expressions accented with $\tilde{}$ and $\tilde{\tilde{}}$ will be subsequently improved. As they are improved once, we remove a $\tilde{}$. For example, $\tilde{\tilde{\mathcal{R}}}_1$ will be improved once, and we use $\tilde{\mathcal{R}}_1$ for the improvement. Then we improve $\tilde{\mathcal{R}}_1$ again to get $\mathcal{R}_1$ which is the final improvement. We will denote $\mathcal{R}_{i}^{0}=\mathcal{R}|_{\mu=0}$ for $i=0, 1, \infty, 2$, and $\mathcal{R'}^{0}=\mathcal{R'}|_{\mu=0}$. Universal constants will be denoted by $c$, and updated throughout the paper. We denote by $P=P_m$ the $L^2$-projection onto the space $H_m$, where
\begin{align}
H_m:= \hbox{span}\{e^{ikx\frac{2\pi}{L}}: 0<|k|\leq m\}. \label{hm}
\end{align}
\section{The Statement Of The Main Result}

We use the Banach spaces,
	\begin{align}
		X=\dot{C}_{b}(\mathbb{R}, P_mH^2)= \{ v:\mathbb{R} \rightarrow  P_mH^2 : v \text{ is continuous and bounded, }\int v=0 \}, \label{X}
	\end{align}
	\begin{align}
		Y=\dot{C}_{b}(\mathbb{R}, H^2)= &\{w: \mathbb{R}\rightarrow H^2 :  w \text{ is continuous, } |w|_{H^{2}} \text{ is bounded, and } \int w=0 \},\label{Y}
	\end{align}
with the following norms,
	\begin{align*}
		|v|_X= \sup_{s\in \mathbb{R}} |v(s)|_{H^{2}}, \qquad |w|_Y= \sup_{s\in \mathbb{R}} |w(s)|_{H^2}.
	\end{align*}	
Let $v\in X$, and consider the equation
	\begin{equation}
		w_s+ww_{x}+w_{xxx}+\gamma w=f- \mu [P_m(w)-v],  \label{KdVw}
	\end{equation}
subject to periodic boundary conditions
	\begin{align}
		\partial_{x}^{j}w(s, x)=\partial_{x}^{j}w(s, x+L), \qquad \forall (s, x)\in \mathbb{R}\times \label{boundaryw}		\mathbb{R}. 	
	\end{align}
for $j=0,1$ and $2$. We assume that $f\in \dot{H}_{\text{per}}^2$. The following theorem is the combination of several theorems in the subsequent sections. The conditions $\eqref{condition1}$, $\eqref{condition2}$, $\eqref{condition3}$,  $\eqref{condition4}$, $\eqref{condition3'}$, $\eqref{condition5}$, $\eqref{condition6}$, and $\eqref{condition4'}$ are defined in subsequent sections.

\newtheorem*{ANA4}{Theorem} \label{MT}
	\begin{ANA4}
		Let $\rho=4R$, where $R$ is given in $\eqref{Rkdv1}$. Let $v\in \mathcal{B}_{X}^{\rho}(0)$, i.e., $v\in X$, with $|v|_{X}<\rho$, and $u^*$ be a steady state solution of equation $\eqref{KdV}$. Then we have the following:
		\begin{enumerate}
			\item Assume that $\mu$ and $m$ are large enough depending on $\rho$, such that $\eqref{condition1}$, $\eqref{condition2}$, $\eqref{condition3}$ and $\eqref{condition4}$ hold, then
				\begin{enumerate}
					\item there exists a unique bounded solution $w\in Y$ of $\eqref{KdVw}$, which holds in the space $H^{-1}$. This defines a map $W: \mathcal{B}_{X}^{\rho}(0)\to Y$, such that $w=W(v)$.
					\item The map $P_mW: \mathcal{B}_{X}^{\rho}(0) \to X$ is a Lipschitz map.
					\item Assume also that $\mu$ and $m$ are large enough such that $\eqref{condition3'}$ holds, then we have $W(P_mu(s))=u(s)$, if and only if $u(\cdot)$ is a trajectory in the global attractor of $\eqref{KdV}$.
					\item The determining form
						\begin{align*}
							\frac{dv}{dt}= F(v)= -|v-P_mW(v)|_{X}^2 (v- P_mu^*),
						\end{align*}
					is an ordinary differential equation which has global existence and uniqueness in the forward invariant set
							$$\{v\in X: |v-P_mu^*|_X< 3R\}.$$
					Furthermore, the trajectory $P_mu(s)$, $s\in \mathbb{R}$, is included in this set, for every $u(\cdot)\subset \mathcal{A}$.
				\end{enumerate}	
			\item Assume that $\mu$ and $m$ are large enough such that $\eqref{condition1}$, $\eqref{condition2}$, $\eqref{condition3'}$, $\eqref{condition5}$ and $\eqref{condition6}$ hold. Let $u(s)$ be the global solution of $\eqref{KdV}-\eqref{boundary}$, for a given initial data $u(s_{0})\in \dot{H}^{2}$, which satisfies  $|u(s)|_{H^{2}}<\rho$, for all $s\geq s_{0}$; and let $w(s)$ be the global solution of  $\eqref{KdVw}-\eqref{boundaryw}$, with $v=P_{m}(u)$, for arbitrary initial data $w(s_{0})\in \dot{H}^{2}$,  where  $|w(s_{0})|_{H^{2}}<\rho$.  Then $|w(s)-u(s)|\rightarrow 0$, as $s\rightarrow \infty$, at an exponential rate with exponent $\frac{\gamma}{4}$.
			\item Assume that $m$ is large enough such that $\eqref{condition4'}$ holds. Then the determining modes property holds for $\eqref{KdV}$, $\eqref{boundary}$, and the number of determining modes, $m$, is of the order of $O(\gamma^{-\frac{26}{3}}, |f|_{H^{2}}^{\frac{14}{3}})$, as $\gamma \rightarrow 0$ and $|f|_{H^{2}}\rightarrow \infty$.
		\end{enumerate}
	\end{ANA4}
We note that the above theorem assumes the existence of a steady state solution to $\eqref{KdV}$, which is proved in Appendix A. For the existence of solution to $\eqref{KdVw}$ which belongs to $Y$, we use parabolic regularization (see, e.g., \cite{Temam1}). Specifically, we consider 
	\begin{equation}
		w_s+\epsilon w_{xxxx}+ww_{x}+w_{xxx}+\gamma w=f- \mu [P_m(w)-v],  \label{KdVwReg}
	\end{equation}
where $0<\epsilon<1$, subject to periodic boundary conditions
	\begin{align*}
		\partial_{x}^{j}w(s, x)=\partial_{x}^{j}w(s, x+L), \qquad \forall (s, x)\in \mathbb{R}\times 			\mathbb{R}. 	
	\end{align*}
for $j=0, 1, 2$. We give the proof using the Galerkin approximation method. Our plan is to find a sufficiently smooth bounded solution $w^{\epsilon}$ to $\eqref{KdVwReg}$ for all $s\in \mathbb{R}$. In order to do that, we consider the Galerkin approximation of $\eqref{KdVwReg}$:
	\begin{equation}
		\frac{\partial w_{n}}{\partial s}+\epsilon (w_{n})_{xxxx}+P_{n}(w_{n}(w_{n})_{x})+(w_{n})_{xxx}+\gamma w_{n}=f_{n}- \mu [P_m(w_{n})-v],  \label{KdVwRegGal}
	\end{equation}
with initial data
	\begin{align*}
		w_n(-k,x)=0, \qquad \text{for some fixed } k\in \mathbb{N},
	\end{align*}
where $w_{n}, f_{n}\in H_n$. Here we consider $n>m$. Note that $P_nv=v$, for every $v\in X$. Since $\eqref{KdVwRegGal}$ is an ordinary differential equation with locally Lipschitz nonlinearity, it has a unique, bounded solution $w^{n}$ on a small interval $[-k, S^*)$, for some $S^*>-k$. We establish here global existence (in time), uniform in $n$ and $s$ bounds (which may depend on $\frac{1}{\epsilon}$) in the norms of the spaces $\dot{L}^2$ and $\dot{H}^4$, which will imply, among other things, that $S^*=\infty$. This is done by taking the inner product of $\eqref{KdVwRegGal}$ with $w^{n}$ and $\partial^{8}_{x}w^{n}$ which are in the space $H_{n}$, respectively, and then using integration by parts, H\"older, Agmon, interpolation and Gronwall inequalities. We denote this solution by $w^{\epsilon}_{n,k}\in C_{b}((-k, \infty); \dot{H}^{4})\cap L^{2}((-k, \infty); \dot{H}^{6})$, referring to its dependence on $\epsilon, n$ and the initial time $-k$. By using the Arzela-Ascoli theorem and Cantor diagonal process on the sequence in $k$, we obtain a solution $w^{\epsilon}_{n}\in C_{b}((-\infty, \infty); \dot{H}^{4})\cap L^{2}((-\infty, \infty); \dot{H}^{6})$ that is defined for all $s\in\mathbb{R}$, and enjoys certain estimates, which are uniform with respect to n. Then, by applying the Aubin compactness theorem (see e.g. \cite{ConFoias}, \cite{TemamDyn}) and Cantor diagonal process to the sequence in $n$, we obtain a bounded solution $w^{\epsilon}\in L^{\infty}((-\infty, \infty); \dot{H}^{4})\cap L^{2}((-\infty, \infty); \dot{H}^{6})$ to $\eqref{KdVwReg}$. Moreover, we show that the equation $\eqref{KdVwReg}$ holds in $L^2(\mathbb{R}, \dot{L}^{2})$. Then, using the Hamiltonian structure of equation $\eqref{KdVwReg}$ when $\epsilon=0$, $f=0$ and $\mu=0$, we find uniform in $\epsilon$ estimates for $w^{\epsilon}$ in $H^{1}$ and $H^{2}$. Finally, we take $\epsilon \rightarrow 0$, and obtain a bounded solution $w\in H^{2}$ on $\mathbb{R}$ to the equation $\eqref{KdVw}$, which obeys the bounds we find in previous steps. In the next section, we establish the above described a priori estimates for $\eqref{KdVwRegGal}$ and steps.

It is worth mentioning that following similar ideas to those in section 4, below, for the case $\mu=0$ in equation $\eqref{KdVwReg}$, and assuming that $w$ is a solution that lies in the global attractor $\mathcal{A}$, one can obtain an explicit bound for $R$ in $\eqref{Rkdv1}$.

\section{A Priori Estimates}

\subsection{$L^2$ bound}
Let $[-k, S^*)$ be the maximal interval of existence for $\eqref{KdVwRegGal}$. For simplicity we will drop the subscripts $n$ and $m$; we will use $w$ instead of $w_{n}$ and $Pw$ instead of $P_{m}w$. Multiply $\eqref{KdVwRegGal}$ by $w$, and integrate over the spatial domain to get
	\begin{align*}
		\int ww_{s}+\epsilon \int w_{xxxx}w+ \int w^{2}w_{x}+\int w_{xxx}w&+\gamma\int w^{2}= \\
													           &\int fw-\mu\int (Pw)^{2}+\mu\int vPw.
	\end{align*}
Here we used the fact that $v$ is in $H_{m}$. We recognize that for any natural number $l$,
	\begin{align}
		\int w^{l}w_{x}=\int \frac{1}{l+1} \partial_{x}(w^{l+1})=0 \label{wlwx}
	\end{align}
Also,
	\begin{align}
		\int w_{xxx}w=-\int w_{xx}w_{x}=-\int \frac{1}{2}\partial_{x}(w_{x})^{2}=0. \label{hoca1}
	\end{align}
Thus, by the Cauchy-Schwarz and Young inequalities, we have
	\begin{align}
		\frac{d}{ds}|w|^{2}+2\epsilon |w_{xx}|^{2}+2\gamma |w|^{2}+2\mu |Pw|^{2}&\leq 2|f||w|+2\mu |v|_{X}|Pw| \label{lazim}\\
										               &\leq \frac{|f|^{2}}{\gamma}+\gamma|w|^{2}+\mu|v|_{X}^{2}+\mu|Pw|^{2},\notag
	\end{align}
or simply
	\begin{align*}
		\frac{d}{ds}|w|^{2}+\gamma |w|^{2}+\mu |Pw|^{2}\leq \frac{|f|^{2}}{\gamma}+\mu|v|_{X}^{2}.
	\end{align*}
Since $w(-k)=0$, and since $|v|_{X}<\rho$, by applying Gronwall's inequality,  we get
	\begin{align*}
		|w(s)|\leq \frac{|f|+\gamma^{\frac{1}{2}}\mu^{\frac{1}{2}}\rho}{\gamma}=: \tilde{\mathcal{R}}_{0}, \text{ for all } s\in [-k, S^{*}).
	\end{align*}
This bound implies the global existence of the Galerkin system $\eqref{KdVwRegGal}$ so that $S^{*}=\infty$. Moreover, we have $w\in C_{b}([-k, \infty); \dot{L}^{2})$.

\subsection{$H^{4}$ bound}
Multiply $\eqref{KdVwRegGal}$ by $\partial_{x}^{8}w\in H_{n}$, integrate over the spatial domain and use integration by parts, the Cauchy-Schwarz and Young inequalities, to obtain
	\begin{align}
		\frac{d}{ds}|\partial_{x}^{4}w|^{2}+2\epsilon |\partial_{x}^{6}w|^{2}&+2\gamma |\partial_{x}^{4}w|^{2}+2\mu |\partial_{x}^{4}Pw|^{2}\leq 2|f_{xx}||\partial_{x}^{6}w|\notag\\
		&+2\mu |v|_{X}|\partial_{x}^{6}Pw|-\int ww_{x}\partial_{x}^{8}w\notag \\
										               &\leq \frac{4}{\epsilon}(|f_{xx}|^{2}+\mu^{2}|v|_{X}^{2})+\frac{\epsilon}{2}|\partial_{x}^{6}w|^{2}-\int ww_{x}\partial_{x}^{8}w.\label{bu}
	\end{align}
Here
	\begin{align*}
		\int ww_{x}\partial^{8}_{x}w&= -\int w_{x}^{2}\partial^{7}_{x}w-\int ww_{xx}\partial^{7}_{x}w\\
						  &= 3\int w_{x}w_{xx}\partial^{6}_{x}w+\int ww_{xxx}\partial^{6}_{x}w\\
						  &\leq 3|w_{x}|_{\infty}|w_{xx}||\partial^{6}_{x}w|+|w|_{\infty}|w_{xxx}||\partial^{6}_{x}w| \\
						  &\leq 3|w_{x}|^{\frac{1}{2}}|w_{xx}|^{\frac{3}{2}}|\partial^{6}_{x}w|+ |w|^{\frac{1}{2}}|w_{x}|^{\frac{1}{2}}|w_{xxx}||\partial^{6}_{x}w|\tag{Agmon's inequality}\\
						  &\leq c|w|^{\frac{17}{12}}|\partial^{6}_{x}w|^{\frac{19}{12}}\tag{Interpolation of $H^{1}, H^{2}, H^{3}$ between $L^{2}$ \& $H^{6}$}\\
						  &\leq c\tilde{R}_{0}^{\frac{17}{12}}|\partial^{6}_{x}w|^{\frac{19}{12}}\\		
						  &\leq \frac{c}{\epsilon^{\frac{19}{5}}}+ \frac{\epsilon}{2} |\partial^{6}_{x}w|^2. \tag{Young's inequality}
	\end{align*}
We use this in $\eqref{bu}$, and absorb the term $ \frac{\epsilon}{2} |\partial^{6}_{x}w|^2$ in the left-hand side, and find an $H^{4}$ bound for $w=w^{\epsilon}_{n,k}$. Thus,
	\begin{align} \label{eqn1}
		w^{\epsilon}_{n,k}\in C_{b}((-k, \infty); \dot{H}^{4})\cap L^{2}_{loc}((-k, \infty); \dot{H}^{6}).
	\end{align}
We realize that since
	\begin{equation*}
		\frac{\partial w^{\epsilon}_{n,k}}{\partial s}=-\epsilon (w^{\epsilon}_{n,k})_{xxxx}-P_{n}(w^{\epsilon}_{n,k}(w^{\epsilon}_{n,k})_{x})-(w^{\epsilon}_{n,k})_{xxx}-\gamma w^{\epsilon}_{n,k}+f_{n}- \mu [P_m(w^{\epsilon}_{n,k})-v],
	\end{equation*}
we have 	(see section 4.7 of \cite{JST})
	\begin{align} \label{eqn2}
		\frac{\partial w^{\epsilon}_{n,k}}{\partial s}\in C_{b}((-k, \infty); \dot{L}^{2})\cap L^{2}_{loc}((-k, \infty); \dot{H}^{2}).	
	\end{align}
We observe that the bounds of the relevant norms in $\eqref{eqn1}$, $\eqref{eqn2}$ are independent of $n$ and $k$. We now focus on the sequence $(w^{\epsilon}_{n,k})_{k\in \mathbb{N}}\subset H_{n}$ for fixed $\epsilon$, and fixed $n$, but $k$ is variable. Since $H_{n}$ is a finite dimensional space, thanks to \eqref{eqn2} we may invoke the Arzela-Ascoli theorem to extract a subsequence $((w^{\epsilon}_{n,k})^{(1)})$ of $(w^{\epsilon}_{n,k})$ such that $(w^{\epsilon}_{n,k})^{(1)}\rightarrow (w^{\epsilon}_{n})^{(1)}$, as $k\rightarrow  \infty$, uniformly on the interval $[-1,1]$. Moreover,   $(w^{\epsilon}_{n})^{(1)}$ is  a bounded solution of $\eqref{KdVwRegGal}$ on the interval $[-1,1]$. Let $j\in \mathbb{N}$, we use an induction iterative procedure to define $(w^{\epsilon}_{n,k})^{(j+1)}$ to be subsequence of $(w^{\epsilon}_{n,k})^{(j)}$, all of which are subsequences of $w^{\epsilon}_{n,k}$. Thanks to \eqref{eqn2} we can apply the Arzela-Ascoli theorem to extract a subsequence of $(w^{\epsilon}_{n,k})^{(j)}$, denoted by $(w^{\epsilon}_{n,k})^{(j+1)}$, such that $(w^{\epsilon}_{n,k})^{(j+1)}\rightarrow (w^{\epsilon}_{n})^{(j+1)}$, as $k\rightarrow \infty$, uniformly on the interval $[-(j+1), (j+1)]$. We notice that $(w^{\epsilon}_{n})^{(j)}$ satisfies all the uniform estimates (see \eqref{eqn1} and \eqref{eqn2}) that are satisfied above by the sequence $w^{\epsilon}_{n,k}$ in the interval $[-j, j]$. Then by the Cantor diagonal process, we obtain that $(w^{\epsilon}_{n,k})^{(k)}\rightarrow w^{\epsilon}_{n}$ on every interval $[-M, M]$, where $w^{\epsilon}_{n}$ is a bounded solution of the Galerkin approximation of the regularized equation $\eqref{KdVwRegGal}$ on all of $\mathbb{R}$ such that
	$$w^{\epsilon}_{n}\in C_{b}((-\infty, \infty); \dot{H}^{4})\cap L^{2}_{loc}((-\infty, \infty); \dot{H}^{6}),$$
and obeys the estimates we found in previous steps. Since
	\begin{equation*}
		\frac{\partial w^{\epsilon}_{n}}{\partial s}=-\epsilon (w^{\epsilon}_{n})_{xxxx}-P_{n}(w^{\epsilon}_{n}(w^{\epsilon}_{n})_{x})-(w^{\epsilon}_{n})_{xxx}-\gamma w^{\epsilon}_{n}+f_{n}- \mu [P_m(w^{\epsilon})-v],
	\end{equation*}
we have
	$$\frac{\partial w^{\epsilon}_{n}}{\partial s}\in C_{b}((-\infty, \infty); \dot{L}^{2})\cap L^{2}_{loc}((-\infty, \infty); \dot{H}^{2}).$$
Now, by using Aubin's compactness theorem and Cantor's diagonal process, one can find
	$$w^{\epsilon}\in L^{\infty}((-\infty, \infty); \dot{H}^{4})\cap L^{2}_{loc}((-\infty, \infty); \dot{H}^{6}),$$
that solves $\eqref{KdVwReg}$, and the equation holds in $L^{\infty}(\mathbb{R}; \dot{L}^{2})\cap L^{2}_{loc}(\mathbb{R}; \dot{H}^{2})$. 

\subsection{$H^1$ bound (uniform in $\epsilon$)}
Next, we establish uniform in $\epsilon$ bounds for $w^{\epsilon}$. For simplicity, we will write $w$ instead of $w^{\epsilon}$. Since $w\in H^{4}$ and $H^{4}$ is an algebra, we have $w^{2}\in H^{4}\subset L^{2}$. Thus, we can take the inner product of equation $\eqref{KdVwReg}$, which holds in $L^{\infty}(\mathbb{R}, L^{2})$, with $2w_{xx}+w^{2}$. Integrating over the spatial domain, using integration by parts, and observing that the function $\Phi$, defined below, is absolutely continuous, we have
	\begin{align}
		\frac{d}{ds}\Phi+\gamma \Phi+2\epsilon|w_{xxx}|^{2}= &-\gamma \int w_{x}^{2}-2\mu\int (Pw_{x})^{2}+\frac{2\gamma}{3}\int w^{3}\notag \\
												  &-2\int f_{xx}w-2\mu\int v_{xx}Pw-\int fw^{2}\notag \\
												  &+\mu\int w^{2}Pw-\mu \int vw^{2}+\epsilon \int w_{xxxx}w^{2}, \label{F1equation*}
	\end{align}
where
	\begin{align}
		\Phi(w)=\int (w_{x}^{2}-\frac{w^{3}}{3}). \label{func1}
	\end{align}
We now estimate the right-hand side of $\eqref{F1equation*}$. Here
	\begin{align*}
		\frac{2\gamma}{3}\int w^{3}+\mu\int w^{2}Pw&\leq (\gamma+\mu)|w|_{\infty}|w|^{2}\\
									              & \leq (\gamma+\mu)|w|^{\frac{5}{2}}|w_{x}|^{\frac{1}{2}}\\
									              & \leq \frac{(\gamma+\mu)^{\frac{4}{3}}}{\gamma^{\frac{1}{3}}}|w|^{\frac{10}{3}}+ \frac{\gamma}{2}|w_{x}|^{2}\\
									              &\leq\frac{(\gamma+\mu)^{\frac{4}{3}}}{\gamma^{\frac{1}{3}}}\tilde{\mathcal{R}}_{0}^{\frac{10}{3}}+ \frac{\gamma}{2}|w_{x}|^{2}.
	\end{align*}
Also,
	\begin{align*}
		&-2\int f_{xx}w\leq2|f|_{H^{2}}|w|\leq 2|f|_{H^{2}}\tilde{\mathcal{R}}_{0} \\
		&-\int fw^{2}-\mu \int vw^{2}\leq (|f|_{\infty}+\mu|v|_{X})|w|^{2}\leq  (|f|_{\infty}+\mu|v|_{X})\tilde{\mathcal{R}}_{0}^{2},\\
		&-2\mu\int v_{xx}Pw\leq 2\mu |v_{xx}||Pw|\leq 2\mu|v|_{X}\tilde{\mathcal{R}}_{0}.
	\end{align*}
For the last term on the right hands side of $\eqref{F1equation*}$, we do the following:
	\begin{align*}
		\epsilon \int w_{xxxx}w^{2}&= -2\epsilon \int ww_{x}w_{xxx}\\
						        &\leq 2\epsilon |w|_{\infty}|w_{x}||w_{xxx}|\\
						        &\leq 2\epsilon |w|^{\frac{1}{2}}|w_{x}|^{\frac{3}{2}}|w_{xxx}|\tag{Agmon}\\
						        &\leq \epsilon c|w|^{\frac{3}{2}}|w_{xxx}|^{\frac{3}{2}} \tag{Interpolation of $H^{1}$ between $L^{2}$ and $H^{3}$}\\
						        &\leq \epsilon c\tilde{\mathcal{R}}_{0}^{\frac{3}{2}}|w_{xxx}|^{\frac{3}{2}}\\
						        &\leq  c \epsilon \tilde {\mathcal{R}}_{0}^{6}+ \frac{3\epsilon}{4}|w_{xxx}|^{2}. \tag{Young}
	\end{align*} 	
Thus, we obtain
\begin{align*}
		\frac{d}{ds}\Phi+\gamma \Phi\leq &\frac{(\gamma+\mu)^{\frac{4}{3}}}{\gamma^{\frac{1}{3}}}\tilde{\mathcal{R}}_{0}^{\frac{10}{3}}+ (|f|_{\infty}+\mu|v|_{X})\tilde{\mathcal{R}}_{0}^{2}+ 2(|f|_{H^{2}}+\mu|v|_{X})\tilde{\mathcal{R}}_{0}+ c \epsilon \tilde {\mathcal{R}}_{0}^{6}.
	\end{align*}
We realize that the right-hand side of the above inequality does not depend on $\epsilon$ (since $\epsilon \in (0, 1)$). Let $s\in \mathbb{R}$, and take $s_{0}<s$. Then by Gronwall's inequality, we obtain
	\begin{align*}
		&\Phi(w(s)) \leq e^{-\gamma(s-s_{0})}\Phi(w(s_{0}))\notag \\
		                 &+\frac{1}{\gamma} \left \{ \frac{(\gamma+\mu)^{\frac{4}{3}}}{\gamma^{\frac{1}{3}}}\tilde{\mathcal{R}}_{0}^{\frac{10}{3}}+ (|f|_{\infty}+\mu|v|_{X})\tilde{\mathcal{R}}_{0}^{2}+ (2|f|_{H^{2}}+2\mu|v|_{X})\tilde{\mathcal{R}}_{0}+c \epsilon \tilde {\mathcal{R}}_{0}^{6}\right \}.
	\end{align*}
Notice that for $\epsilon>0$, we have $w\in L^{\infty}(\mathbb{R}, H^{4})$, in particular $\Phi(w(s_{0}))$ is bounded uniformly, for every $s_{0}\in \mathbb{R}$. Now we let $s_{0}\rightarrow -\infty$. Since $e^{-\gamma(s-s_{0})}\Phi(w(s_{0}))\rightarrow 0$, and since $|v|_{X}<\rho$, we have
	\begin{align}
		\Phi(w(s)) \leq \frac{1}{\gamma} \left \{ \frac{(\gamma+\mu)^{\frac{4}{3}}}{\gamma^{\frac{1}{3}}}\tilde{\mathcal{R}}_{0}^{\frac{10}{3}}+ (|f|_{\infty}+\mu\rho)\tilde{\mathcal{R}}_{0}^{2}+ (2|f|_{H^{2}}+2\mu\rho)\tilde{\mathcal{R}}_{0}+c \epsilon \tilde {\mathcal{R}}_{0}^{6}\right \}, \label{ineq1*}
	\end{align}
for all $s\in \mathbb{R}$. Also,
	\begin{align*}
		\Phi(w)=\int w_{x}^{2}-\frac{w^{3}}{3}&\geq |w_{x}|^{2}-\frac{1}{3}|w|^{\frac{5}{2}}|w_{x}|^{\frac{1}{2}}\\
									&\geq |w_{x}|^{2}-|w|^{\frac{10}{3}}-\frac{1}{2}|w_{x}|^{2},
	\end{align*}
and hence
	\begin{align}
		|w_{x}|^{2}&\leq 2\Phi(w)+2|w|^{\frac{10}{3}}\leq 2\Phi(w)+2\tilde{\mathcal{R}}_{0}^{\frac{10}{3}}. \label{hoca2}
	\end{align}
Using $\eqref{ineq1*}$, we obtain that
		\begin{align*}
		|w_{x}| \leq \tilde{\tilde{\mathcal{R}}}_{1},
	\end{align*} 	
where
	\begin{align}
		\tilde{\tilde{\mathcal{R}}}_{1}^{2}:=&\frac{2\left((\gamma+\mu)^{\frac{4}{3}}+\gamma^{\frac{4}{3}}\right)}{\gamma^{\frac{4}{3}}}\tilde{\mathcal{R}}_{0}^{\frac{10}{3}}\notag \\
		&+ \frac{2}{\gamma}\left \{(|f|_{\infty}+\mu\rho)\tilde{\mathcal{R}}_{0}^{2}+ 2(|f|_{H^{2}}+\mu\rho)\tilde{\mathcal{R}}_{0}+c \epsilon \tilde {\mathcal{R}}_{0}^{6}\right \}. \label{wxtt}
	\end{align}
Thus, we have a $H^{1}$ bound which is uniform in $\epsilon$, since $\epsilon\in (0, 1)$. Note that $\tilde{\tilde{\mathcal{R}}}_{1}=O(\mu^{\frac{3}{2}})$ as $\mu \rightarrow \infty$.

\subsection{Improved $L^{2}$ bound}
We will use $\tilde{\mathcal{R}}_{1}$ to get a sharper bound for $|w|$. We define the operator $Q:= I-P$, i.e., $Qw=w-Pw$. We add $2\mu|Qw|^{2}$ to both sides of $\eqref{lazim}$, use Young's and Poincar\'e's  inequalities to get
	\begin{align*}
		\frac{d}{ds}|w|^{2}+2\epsilon |w_{xx}|^{2}+2\gamma |w|^{2}+2\mu |w|^{2}&\leq 2|f||w|+2\mu |v|_{X}|Pw|+2\mu |Qw|^{2}\\		
											   &\leq \frac{|f|^{2}}{\gamma}+\gamma|w|^{2}+\mu|v|_{X}^{2}+\mu|Pw|^{2}+ \frac{\mu L^{2}}{2\pi^{2} (m+1)^{2}}|w_{x}|^{2}\\
											   &\leq \frac{|f|^{2}}{\gamma}+\mu|v|_{X}^{2}+ (\gamma+\mu)|w|^{2}+ \frac{\mu L^{2}}{2\pi^{2} (m+1)^{2}}\tilde{\tilde{\mathcal{R}}}_{1}^{2}.						              
	\end{align*}
Now, if we choose $m$ large enough such that
	\begin{align}
		\frac{\mu L^{2}}{2\pi^{2} (m+1)^{2}}\tilde{\tilde{\mathcal{R}}}_{1}^{2}\leq \mu^{\alpha}, \label{condition1}
	\end{align}
for some $\alpha\geq 1$, to be determined later, since $|v|_{X}<\rho$, then
	\begin{align*}
		\frac{d}{ds}|w|^{2}+(\gamma+\mu) |w|^{2}&\leq \frac{|f|^{2}}{\gamma}+\mu(\rho^{2}+\mu^{\alpha-1}). 				              
	\end{align*}
By Gronwall's inequality on the interval $[s_0,s]$, and letting $s_0 \to - \infty$, we obtain
	\begin{align}
		|w(s)|\leq \sqrt{\frac{\frac{|f|^{2}}{\gamma}+\mu(\rho^{2}+\mu^{\alpha-1})}{(\gamma+\mu)}}&\leq \sqrt{\frac{|f|^{2}}{\gamma^{2}}+\rho^{2}+\mu^{\alpha-1}}\notag\\
																			 &\leq \frac{|f|}{\gamma}+\rho+\mu^{\frac{\alpha-1}{2}} =: \mathcal{R}_{0}. \label{w}			
	\end{align}
for all $s\in \mathbb{R}$. We observe that $\mathcal{R}_{0}= O(\mu^{\frac{\alpha-1}{2}}, \gamma^{-1}, |f|)$, for $\alpha\in [1, 2)$, as $\mu \rightarrow \infty$, $\gamma \rightarrow 0$ and $|f| \rightarrow \infty$. We chose $\alpha$ less than $2$, so that this is an improvement over $\tilde{\mathcal{R}}_{0}=O(\mu^{\frac{1}{2}})$.

\subsection{Improved $H^{1}$ bound (uniform in $\epsilon$)}
Replacing $\tilde{\mathcal{R}}_{0}$ with $\mathcal{R}_{0}$ in $\eqref{wxtt}$, we find an improved bound
	\begin{align*}
		\sup_{s\in \mathbb{R}}|w_{x}(s)| \leq \tilde{\mathcal{R}_{1}},
	\end{align*}
where
	\begin{align*}
		\tilde{\mathcal{R}}_{1}^{2}:= &\frac{2}{\gamma}\left \{ (\frac{(\gamma+\mu)^{\frac{4}{3}}}{\gamma^{\frac{1}{3}}}+\gamma)\mathcal{R}_{0}^{\frac{10}{3}}+ (|f|_{\infty}+\mu\rho)\mathcal{R}_{0}^{2}+ 2(|f|_{H^{2}}+\mu\rho)\mathcal{R}_{0}+c \epsilon \tilde {\mathcal{R}}_{0}^{6}\right \}.
	\end{align*}
We note that $\tilde{\mathcal{R}}_{1}=O(\mu^{\frac{5\alpha-1}{6}})$, as $\mu \rightarrow \infty$, where $\alpha\in [1, 2)$. 	

\subsection{$H^2$ bound (uniform in $\epsilon$)}
Since $\eqref{KdVwReg}$ holds in $L^{\infty}(\mathbb{R}, \dot{L}^{2})$ and $w\in \dot{H}^{4}$, we multiply $\eqref{KdVwReg}$ by $\frac{18}{5}w_{xxxx}+6ww_{xx}+3w_{x}^{2}+w^{3}\in L^{2}$, and integrate over the spatial domain. From multiplying $\eqref{KdVwReg}$ by the term $\frac{18}{5}w_{xxxx}$, we get
	\begin{align}
		\frac{d}{ds}\int \frac{9}{5}w_{xx}^{2} &+\frac{18\epsilon}{5}|w_{xxxx}|^{2}+ \frac{18}{5}\int ww_{x}w_{xxxx}\notag \\
		&+ \frac{18}{5}\int w_{xxx}w_{xxxx}+\gamma \frac{18}{5}\int w_{xx}^{2}+\mu \frac{18}{5}\int Pw_{xx}^{2}\notag \\
															      &=  \frac{18}{5}\int f_{xx}w_{xx}+\mu  \frac{18}{5}\int v_{xx}w_{xx}. \label{wxx1}
	\end{align}
We have
	\begin{align*}
		 \frac{18}{5}\int ww_{x}w_{xxxx}= - \frac{18}{5}\int w_{x}^{2}w_{xxx}-  \frac{18}{5}\int ww_{xx}w_{xxx},
	\end{align*}
where
	\begin{align*}
		 - \frac{18}{5}\int w_{x}^{2}w_{xxx}= \frac{36}{5}\int w_{x}w_{xx}^{2}, 
	\end{align*}
and
	\begin{align*}
		 -\frac{18}{5}\int ww_{xx}w_{xxx}= \frac{18}{5}\int w_{x}w_{xx}^{2}+\frac{18}{5}\int ww_{xxx}w_{xx},
	\end{align*}
which gives
	\begin{align*}
		 -\frac{18}{5}\int ww_{xx}w_{xxx}= \frac{9}{5}\int w_{x}w_{xx}^{2}.
	\end{align*}
We rewrite $\eqref{wxx1}$ as,
	\begin{align}
		\frac{d}{ds}\int \frac{9}{5}w_{xx}^{2}& +\frac{18\epsilon}{5}|w_{xxxx}|^{2}+9\int w_{x}w_{xx}^{2}+\gamma \frac{18}{5}\int w_{xx}^{2}+\mu \frac{18}{5}\int Pw_{xx}^{2}\notag\\
															      &=  \frac{18}{5}\int f_{xx}w_{xx}+\mu  \frac{18}{5}\int v_{xx}w_{xx}. \label{wxx2}
	\end{align}
From multiplying $\eqref{KdVwReg}$ by the term $6ww_{xx}$, we get
	\begin{align}
		6\int w_{s}ww_{xx}&+6\epsilon \int ww_{xx}w_{xxxx}+6\int w^{2}w_{x}w_{xx}\notag \\
		&+6\int ww_{xx}w_{xxx}+6\gamma \int w^{2}w_{xx}+6\mu \int (Pw)ww_{xx}\notag\\
										   &= 6\int fww_{xx}+6\mu \int vww_{xx} \label{wxx3}.
	\end{align}
Since
	\begin{align*}
		6\int w_{s}ww_{xx}=-6\int w_{sx}ww_{x}-6\int w_{s}w_{x}^{2},
	\end{align*}
and
	\begin{align*}
		\frac{d}{ds}\int ww_{x}^{2}= \int w_{s}w_{x}^{2}+2\int ww_{x}w_{xs},
	\end{align*}
we have that
	\begin{align}
		6\int w_{s}ww_{xx}=-\frac{d}{ds}\int 3ww_{x}^{2}-3\int w_{s}w_{x}^{2}.  \label{wxx5}
	\end{align}
For the remaining terms in $\eqref{wxx3}$, we write
	\begin{align*}
		&6\int w^{2}w_{x}w_{xx}= 3\int w^{2}\frac{\partial}{\partial x}(w_{x}^{2})= -3\int \frac{\partial}{\partial x}(w^{2})w_{x}^{2}= -6\int ww_{x}^{3}, \\
		&6\int ww_{xx}w_{xxx}= 3\int w\frac{\partial}{\partial x}(w_{xx}^{2})= -3\int w_{x}w_{xx}^{2}, \\
		&6\gamma \int w^{2}w_{xx}= -12\gamma \int ww_{x}^{2}, \\
		&6\mu\int (Pw)ww_{xx}=-6\mu\int (Pw_{x})ww_{x}-6\mu\int(Pw)w_{x}^{2}.
	\end{align*}
By using these equalities and $\eqref{wxx5}$ in $\eqref{wxx3}$, we get
	\begin{align}
		-\frac{d}{ds}\int 3ww_{x}^{2}&-3\int w_{s}w_{x}^{2} -6\int ww_{x}^{3}-3\int w_{x}w_{xx}^{2}-12\gamma \int ww_{x}^{2} \notag\\
									                       &-6\mu\int (Pw_{x})ww_{x}-6\mu\int(Pw)w_{x}^{2}\notag\\
										   	    &= 6\int fww_{xx}+6\mu \int vww_{xx}-6\epsilon \int ww_{xx}w_{xxxx} \label{wxx6}.
	\end{align}
Multiplying $\eqref{KdVwReg}$ by the term $3w_{x}^{2}$, we get
	\begin{align}
		3\int w_{s}w_{x}^{2}+3\int ww_{x}^{3}-&6\int w_{x}w_{xx}^{2}+3\gamma\int ww_{x}^{2}+3\mu\int (Pw)w_{x}^{2}\notag\\
									    &= 3\int fw_{x}^{2}+3\mu \int vw_{x}^{2}-3\epsilon \int w_{x}^{2}w_{xxxx}, \label{wxx7}
	\end{align}
since
	\begin{align*}
		3\int w_{x}^{2}w_{xxx}= -6\int w_{x}w_{xx}^{2}.
	\end{align*}
Multiplying $\eqref{KdVwReg}$ by the term $w^{3}$, we get
	\begin{align}
		\int w_{s}w^{3}+\epsilon \int w^{3}w_{xxxx}+\int w^{4}w_{x}+&\int w^{3}w_{xxx}+\gamma\int w^{4}+\mu \int (Pw)w^{3}\notag \\
								 &= \int fw^{3}+\mu \int vw^{3}. \label{wxx8}
	\end{align}
Since
	\begin{align*}
		&\int w_{s}w^{3}= \frac{d}{ds}\int \frac{1}{4}w^{4}, \text{  } \int w^{4}w_{x}= 0, \\
		&\int w^{3}w_{xxx}= -3\int w^{2}w_{x}w_{xx}= -\frac{3}{2}\int w^{2}\frac{\partial}{\partial x}(w_{x}^{2})= \frac{3}{2}\int \frac{\partial}{\partial x}(w^{2})w_{x}^{2}= 3\int ww_{x}^{3},
	\end{align*}
we may rewrite $\eqref{wxx8}$ as
	\begin{align}
		\frac{d}{ds}\int \frac{1}{4}w^{4}+&3\int ww_{x}^{3}+\gamma\int w^{4}+\mu \int (Pw)w^{3}\notag \\
								 &= \int fw^{3}+\mu \int vw^{3}-\epsilon \int w^{3}w_{xxxx}. \label{wxx9}
	\end{align}
Now we add $\eqref{wxx2}$, $\eqref{wxx6}$, $\eqref{wxx7}$ and $\eqref{wxx9}$ to get the following energy equation
	\begin{align*}
		\frac{d}{ds}\int \left(\frac{9}{5}w_{xx}^{2}-3ww_{x}^{2}+\frac{w^{4}}{4}\right)&+\gamma \int \left(\frac{18}{5}w_{xx}^{2}-9ww_{x}^{2}+w^{4}\right)+\frac{18\epsilon}{5}|w_{xxxx}|^{2}\\	
											  =& 	\int \left(\frac{18}{5}f_{xx}w_{xx}+6wfw_{xx}+3fw_{x}^{2}+fw^{3}\right) \\
											  &+ 	\mu \int \left(\frac{18}{5}v_{xx}w_{xx}+6wvw_{xx}+3vw_{x}^{2}+vw^{3}\right)\\
											  &-  	\mu \frac{18}{5} \int (Pw_{xx})^{2}+6\mu \int (Pw_{x})ww_{x}\\
											  &+3\mu \int (Pw)w_{x}^{2}-\mu \int (Pw) w^{3}\\
											  &-6\epsilon \int ww_{xx}w_{xxxx}-3\epsilon\int w_{x}^{2}w_{xxxx}-\epsilon \int w^{3}w_{xxxx}.           
	\end{align*}
We define
	 \begin{align}
		\varphi(w)=\int \left(\frac{9}{5}w_{xx}^{2}-3ww_{x}^{2}+\frac{w^{4}}{4}\right), \label{F2} 
	\end{align}
which is an absolute continuous function, then the above implies
	\begin{align*}
		\frac{d}{ds}\varphi+\gamma \varphi+\frac{18\epsilon}{5}|w_{xxxx}|^{2} = &-\gamma \int \left(\frac{9}{5}w_{xx}^{2}-6ww_{x}^{2}+\frac{3}{4}w^{4}\right)\\
									&+\int \left(\frac{18}{5}f_{xx}w_{xx}+6wfw_{xx}+3fw_{x}^{2}+fw^{3}\right) \\
									&+ \mu \int \left(\frac{18}{5}v_{xx}w_{xx}+6wvw_{xx}+3vw_{x}^{2}+vw^{3}\right)\\
								         &- \mu \frac{18}{5} \int (Pw_{xx})^{2}+6\mu \int (Pw_{x})ww_{x}\\
									&+3\mu \int (Pw)w_{x}^{2}-\mu \int (Pw) w^{3}\\
									&-6\epsilon \int ww_{xx}w_{xxxx}-3\epsilon\int w_{x}^{2}w_{xxxx}-\epsilon \int w^{3}w_{xxxx}.    
	\end{align*}
We now use Young's, H\"older's, Agmon's and interpolation inequalities along with, uniform in $\epsilon\in (0, 1)$, $L^{2}$ and $H^{1}$ bounds for each term on the right-hand side to obtain bounds in which the power on $\mu$ is minimal: 
	\begin{align*}
		6\gamma \int ww_{x}^{2}= 6\gamma \int ww_{x}w_{x}&=3\gamma \int \frac{\partial}{\partial x}(w^{2})w_{x}= -3\gamma \int w^{2}w_{xx}\\
												  &\leq 3\gamma |w|_{\infty}|w||w_{xx}|\leq 3\gamma |w|^{\frac{3}{2}}|w_{x}|^{\frac{1}{2}}|w_{xx}|\\
												  &\leq 3\gamma \mathcal{R}_{0}^{\frac{3}{2}}\tilde{\mathcal{R}}_{1}^{\frac{1}{2}}|w_{xx}|,
	\end{align*}
	\begin{align*}
		3 \int (f+\mu v)w_{x}^{2}&= 3 \int (f+\mu v)w_{x}w_{x}=-3 \int (f_{x}+\mu v_{x})w_{x}w- 3 \int (f+\mu v)w_{xx}w\\
								 &= \frac{3}{2} \int (f_{xx}+\mu v_{xx})w^{2}-3 \int (f+\mu v)ww_{xx}\\
							          &\leq \frac{3}{2} |w|_{\infty}|w|(|f|_{H^{2}}+\mu |v|_{X})+ 3 (|f|_{\infty}+\mu |v|_{\infty})|w||w_{xx}|\\	
							          &\leq  \frac{3}{2} \mathcal{R}_{0}^{\frac{3}{2}}\tilde{\mathcal{R}}_{1}^{\frac{1}{2}}(|f|_{H^{2}}+\mu |v|_{X})+3 (|f|_{\infty}+\mu |v|_{\infty})\mathcal{R}_{0}|w_{xx}|,   								       									     
	\end{align*}
	 \begin{align*}
		6\mu \int (Pw_{x})ww_{x}+3\mu \int (Pw)w_{x}^{2}&= -\frac{3}{2}\mu \int (Pw_{xx})w^{2}-3\mu \int Pw ww_{xx}\\
											     &\leq \frac{9}{2}\mu\mathcal{R}_{0}^{\frac{3}{2}}\tilde{\mathcal{R}}_{1}^{\frac{1}{2}} |w_{xx}|,
	\end{align*}
	\begin{align*}
		   \int \frac{18}{5}f_{xx}w_{xx}+ \mu \int \frac{18}{5}v_{xx}w_{xx}\leq  \frac{18}{5}(|f|_{H^{2}}+\mu |v|_{X})|w_{xx}|,
	\end{align*}
	\begin{align*}
		  \int 6wfw_{xx}+\int 6w\mu vw_{xx}&\leq 6(|f|_{\infty}+\mu|v|_{\infty})|w||w_{xx}|\\
		   						       &\leq 6(|f|_{H^{2}}+\mu|v|_{X})\mathcal{R}_{0}|w_{xx}|,
	\end{align*}
	\begin{align*}
		\int fw^{3}+ \int \mu vw^{3}-\mu \int (Pw) w^{3}\leq &(|f|_{H^{2}}+\mu |v|_{X})\mathcal{R}_{0}^{\frac{5}{2}}\tilde{\mathcal{R}}_{1}^{\frac{1}{2}}+ \mu \mathcal{R}_{0}^{3}\tilde{\mathcal{R}}_{1}.
	\end{align*}
For the terms with $\epsilon$, we have
	\begin{align*}
		&-6\epsilon \int ww_{xx}w_{xxxx}-3\epsilon\int w_{x}^{2}w_{xxxx}-\epsilon \int w^{3}w_{xxxx}\\
								  &\leq 6\epsilon |w|_{\infty}|w_{xx}||w_{xxxx}|+ 3\epsilon |w_{x}|_{\infty}|w_{x}||w_{xxxx}|+ \epsilon |w|_{\infty}^{2}|w||w_{xxxx}|\\
								  &\leq  6\epsilon |w|^{\frac{1}{2}}|w_{x}|^{\frac{1}{2}}|w_{xx}||w_{xxxx}|+ 3\epsilon |w_{x}|^{\frac{3}{2}}|w_{xx}|^{\frac{1}{2}}|w_{xxxx}|+\epsilon |w|^{2}|w_{x}||w_{xxxx}| \tag{Agmon}\\
								  &\leq  9\epsilon c|w|^{\frac{11}{8}}|w_{xxxx}|^{\frac{13}{8}}+\epsilon c|w|^{\frac{11}{4}}|w_{xxxx}|^{\frac{5}{4}} \tag{Interpolation}\\
								  &\leq \epsilon c |w|^{\frac{22}{3}}+ \frac{23\epsilon}{16}|w_{xxxx}|^{2} \tag{Young}\\
								  &\leq  \epsilon c \mathcal{R}_{0}^{\frac{22}{3}}+ \frac{18\epsilon}{5}|w_{xxxx}|^{2}.
	\end{align*}
We sum all of above, and use the fact that $|v|_{X}<\rho$, to get 		
	\begin{align*}
		\frac{d}{ds}\varphi+\gamma \varphi \leq -\frac{9}{5}\gamma|w_{xx}|^{2}+\tilde{C}_{1}|w_{xx}| +\tilde{C_{2}}+ \epsilon c \mathcal{R}_{0}^{\frac{22}{3}},
	\end{align*}
where
	\begin{align}
		\tilde{C}_{1}:= &3\gamma \mathcal{R}_{0}^{\frac{3}{2}}\tilde{\mathcal{R}}_{1}^{\frac{1}{2}}+3 (|f|_{\infty}+\mu \rho)\mathcal{R}_{0}+\frac{9}{2}\mu \mathcal{R}_{0}^{\frac{3}{2}}\tilde{\mathcal{R}}_{1}^{\frac{1}{2}}\notag\\
				       &+ \frac{18}{5}(|f|_{H^{2}}+\mu \rho)+6(|f|_{H^{2}}+\mu\rho)\mathcal{R}_{0},\label{c1}
	\end{align}
and
	\begin{align}
		\tilde{C}_{2}:=& \frac{3}{2} \mathcal{R}_{0}^{\frac{3}{2}}\tilde{\mathcal{R}}_{1}^{\frac{1}{2}}(|f|_{H^{2}}+\mu \rho)+ (|f|_{H^{2}}+\mu \rho)\mathcal{R}_{0}^{\frac{5}{2}}\tilde{\mathcal{R}}_{1}^{\frac{1}{2}}+ \mu \mathcal{R}_{0}^{3}\tilde{\mathcal{R}}_{1}. \label{c2}
	\end{align}	
We note that $\tilde{C}_{1}=O(\mu^{\frac{7\alpha+1}{6}})$ and $\tilde{C}_{2}=O(\mu^{\frac{7\alpha-2}{3}})$, for $\alpha\in [1, 2)$. We use Young's inequality to get
	\begin{align*}
		\frac{d}{ds}\varphi+\gamma \varphi \leq \frac{5}{36\gamma}\tilde{C}_{1}^{2} +\tilde{C_{2}}+ \epsilon c \mathcal{R}_{0}^{\frac{22}{3}}.
	\end{align*}
Let $s\in \mathbb{R}$, and take $s_{0}<s$. Then by Gronwall's inequality, we obtain
	\begin{align*}
		\varphi(w(s)) \leq e^{-\gamma(s-s_{0})}\varphi(w(s_{0}))+\frac{5}{36\gamma^{2}}\tilde{C}_{1}^{2} +\frac{1}{\gamma}\tilde{C_{2}}+\frac{\epsilon c \mathcal{R}_{0}^{\frac{22}{3}}}{\gamma}.
	\end{align*}
Notice that for $\epsilon\in (0,1)$, we have $w\in L^{\infty}(\mathbb{R}, \dot{H}^{4})$, in particular $\varphi(w(s_{0}))$ is bounded uniformly, for every $s_{0}$. Now we let $s_{0}\rightarrow -\infty$. Since $e^{-\gamma(s-s_{0})}\varphi(w(s_{0}))\rightarrow 0$, we have
	\begin{align*}
		\varphi(w(s)) \leq \frac{5}{36\gamma^{2}}\tilde{C}_{1}^{2} +\frac{1}{\gamma}\tilde{C_{2}}+\frac{\epsilon c \mathcal{R}_{0}^{\frac{22}{3}}}{\gamma}, \text{ for all } s\in \mathbb{R}.          
	\end{align*}
Since
	\begin{align*}
		3\left |\int ww_{x}^{2}\right |=\frac{3}{2}\left |\int (w^{2})_{x}w_{x}\right |=\frac{3}{2}\left |\int w^{2}w_{xx}\right |\leq \frac{3}{2}\mathcal{R}_{0}^{\frac{3}{2}}\tilde{\mathcal{R}}_{1}^{\frac{1}{2}}|w_{xx}|,
	\end{align*}
we have
	\begin{align*}
		\varphi(w(s)) \geq \frac{9}{5} |w_{xx}|^{2}- \frac{3}{2}\mathcal{R}_{0}^{\frac{3}{2}}\tilde{\mathcal{R}}_{1}^{\frac{1}{2}}|w_{xx}|\geq |w_{xx}|^{2}- \frac{45}{64}\mathcal{R}_{0}^{3}\tilde{\mathcal{R}}_{1},
	\end{align*}
and hence
	\begin{align*}
		|w_{xx}(s))|^{2}\leq \varphi(w(s))+\frac{45}{64}\mathcal{R}_{0}^{3}\tilde{\mathcal{R}}_{1},  \text{ for all } s\in \mathbb{R}.
	\end{align*}
Thus
	\begin{align*}
		|w_{xx}(s)|\leq \sqrt{\frac{5}{36\gamma^{2}}\tilde{C}_{1}^{2} +\frac{1}{\gamma}\tilde{C}_{2}+\frac{\epsilon c \mathcal{R}_{0}^{\frac{22}{3}}}{\gamma}+\frac{45}{64}\mathcal{R}_{0}^{3}\tilde{\mathcal{R}}_{1}}=:\tilde{\mathcal{R}}_{2},  \text{ for all } s\in \mathbb{R}.
	\end{align*}
Note that $\tilde{\mathcal{R}}_{2}=O(\mu^{\frac{7\alpha+1}{6}})$, as $\mu \rightarrow \infty$, and that $\tilde{\mathcal{R}}_{2}$ is bounded uniformly in $\epsilon$, since $\epsilon\in (0, 1)$.

\subsection{More improved $H^{1}$ bound (uniform in $\epsilon$)}

We rewrite $\eqref{F1equation*}$ as
	\begin{align*}
		\frac{d}{ds}\Phi+(\gamma+\mu) \Phi+2\epsilon|w_{xxx}|^{2}= &-\gamma \int w_{x}^{2}-2\mu\int (Pw_{x})^{2}+\frac{2\gamma}{3}\int w^{3}-2\int f_{xx}w\\
							       &-2\mu\int v_{xx}Pw-\int fw^{2}+\mu\int w^{2}Pw-\mu \int vw^{2}\\
							       &+ \mu \int w_{x}^{2}- \frac{\mu}{3} \int w^{3}-\epsilon \int w_{xxxx}w^{2},
	\end{align*}
and update the bounds for the terms on the right-hand side. Here again, we use interpolation inequalities to eliminate the terms with $\epsilon$. Using Poincar\'e's  inequality, we have
	\begin{align*}
		\mu \int w_{x}^{2}= \mu |w_{x}|^{2}&=\mu |Pw_{x}|^{2}+\mu |Qw_{x}|^{2}\\
								     &\leq \mu |Pw_{x}|^{2} + \frac{\mu L^{2}}{4\pi^{2}(m+1)^{2}} |w_{xx}|^{2}\\
								     &\leq  \mu |Pw_{x}|^{2} + \frac{\mu L^{2}\tilde{\mathcal{R}}_{2}^{2}}{4\pi^{2}(m+1)^{2}}.
	\end{align*}
We assume that $m$ is large enough such that
	\begin{align}
		\frac{\mu L^{2}\tilde{\mathcal{R}}_{2}^{2}}{4\pi^{2}(m+1)^{2}}\leq \mu^{\beta}, \label{condition2}
	\end{align}
for some $\beta >0$, to be determined later. Then, since $|f|_{\infty}\leq |f|_{H^{2}}$ and $|v|_{X}<\rho$, we have
	\begin{align*}
		\frac{d}{ds}\Phi+(\gamma+\mu) \Phi\leq &\frac{(\gamma+\mu)^{\frac{4}{3}}}{\gamma^{\frac{1}{3}}}\mathcal{R}_{0}^{\frac{10}{3}}\\
								   &+ (|f|_{H^{2}}+\mu\rho)\mathcal{R}_{0}^{2}+ 2(|f|_{H^{2}}+\mu\rho)\mathcal{R}_{0}+c \epsilon \mathcal{R}_{0}^{6}+\mu^{\beta}.
	\end{align*}
Thus, by $\eqref{hoca2}$, $|w_{x}|\leq \mathcal{R}_{1}$ where
	\begin{align}
		\mathcal{R}_{1}^{2}:=& \frac{2}{\gamma+\mu}\left\{\left(\frac{(\gamma+\mu)^{\frac{4}{3}}}{\gamma^{\frac{1}{3}}}+(\gamma+\mu)\right)\mathcal{R}_{0}^{\frac{10}{3}}\right\}\notag\\
						 &+\frac{2}{\gamma+\mu}\left\{(|f|_{H^{2}}+\mu\rho)\mathcal{R}_{0}^{2}+ (2|f|_{H^{2}}+2\mu\rho)\mathcal{R}_{0}+c \epsilon \mathcal{R}_{0}^{6}+\mu^{\beta}\right\}. \label{wx}
	\end{align}
We see that
	\[ \mathcal{R}_{1} = \begin{cases}
      		 O(\mu^{\frac{\beta-1}{2}}) & \textrm{ if $5\alpha\leq 3\beta+1$,} \\
      		 O(\mu^{\frac{5\alpha-4}{6}})& \textrm{ if $5\alpha> 3\beta+1$,} \\
   	\end{cases} \]
as $\mu \rightarrow \infty$, and $\mathcal{R}_{1}=O(\gamma^{-\frac{5}{3}}, |f|_{H^{2}}^{\frac{5}{3}})$, as $\gamma\rightarrow 0$ and $|f|_{H^{2}}\rightarrow \infty$. 	

\subsection{Improved $H^{2}$ bound (uniform in $\epsilon$)}
We make similar estimates as we did above for finding $H^{2}$ bound using the new $H^{1}$ bound $\mathcal{R}_{1}$. We obtain that
	\begin{align} \label{rtwo}
		\sup_{s\in \mathbb{R}}|w_{xx}(s)|\leq \sqrt{\frac{5}{36\gamma^{2}}C_{1}^{2} +\frac{1}{\gamma}C_{2}+\frac{\epsilon c \mathcal{R}_{0}^{\frac{22}{3}}}{\gamma}+\frac{45}{64}\mathcal{R}_{0}^{3}\mathcal{R}_{1}}=:\mathcal{R}_{2}. 	
	\end{align}
where $C_{1}$, $C_{2}$ are as in $\eqref{c1}$, $\eqref{c2}$ but with $\tilde{\mathcal{R}}_{1}$ replaced by $\mathcal{R}_{1}$. We find that
	\[ C_{1} = \begin{cases}
      		O(\mu^{\frac{3\alpha+\beta}{4}}) & \textrm{ if $5\alpha\leq 3\beta+1$,} \\
      		O(\mu^{\frac{14\alpha-1}{12}}) & \textrm{ if $5\alpha> 3\beta+1$,} \\
   	\end{cases} \]
	\[ C_{2} = \begin{cases}
      		O(\mu^{\frac{3\alpha+\beta-2}{2}}) & \textrm{ if $5\alpha\leq 3\beta+1$,} \\
      		O(\mu^{\frac{14\alpha-7}{6}}) & \textrm{ if $5\alpha> 3\beta+1$,} \\
   	\end{cases} \]
where $C_{1}=O(\gamma^{-\frac{7}{3}}, |f|_{H^{2}}^{\frac{7}{3}})$ and $C_{2}=O(\gamma^{-\frac{14}{3}}, |f|_{H^{2}}^{\frac{14}{3}})$ . Thus, for $\alpha\in [1,2)$,
 	\[ \mathcal{R}_{2} = \begin{cases}
      		O(\mu^{\frac{3\alpha+\beta}{4}}) & \textrm{ if $5\alpha\leq 3\beta+1$,} \\
      		O(\mu^{\frac{14\alpha-1}{12}}) & \textrm{ if $5\alpha> 3\beta+1$,} \\
   	\end{cases} \]
as $\mu \rightarrow \infty$, and $\mathcal{R}_{2}=O(\gamma^{-\frac{10}{3}}, |f|_{H^{2}}^{\frac{7}{3}})$, as $\gamma\rightarrow 0$ and $|f|_{H^{2}}\rightarrow \infty$.	

\subsection{$L^{\infty}$ bound (uniform in $\epsilon$)}
Using Agmon's inequality $\eqref{Agmon}$, we find that
	\begin{align*}
		|w(s)|_{\infty}^{2}\leq |w(s)||w_{x}(s)|\leq \mathcal{R}_{0}\mathcal{R}_{1}, \text{ for all } s\in \mathbb{R}.
	\end{align*}
Thus
	\begin{align*}
		\sup_{s\in \mathbb{R}}|w(s)|_{\infty}\leq \mathcal{R}_{\infty},
	\end{align*}
where
	\begin{align}
		\mathcal{R}_{\infty}:= \mathcal{R}_{0}^{\frac{1}{2}}\mathcal{R}_{1}^{\frac{1}{2}}. \label{winfty}
	\end{align}	
We observe that, for $\alpha\in [1, 2)$,
	\[ \mathcal{R}_{\infty} = \begin{cases}
      		O(\mu^{\frac{\alpha+\beta-2}{4}}) & \textrm{ if $5\alpha\leq 3\beta+1$,} \\
      		O(\mu^{\frac{8\alpha-7}{12}}) & \textrm{ if $5\alpha> 3\beta+1$,} \\
   	\end{cases} \]
as $\mu \rightarrow \infty$, and $\mathcal{R}_{\infty}=O(\gamma^{-\frac{4}{3}}, |f|_{H^{2}}^{\frac{4}{3}})$, as $\gamma\rightarrow 0$ and $|f|_{H^{2}}\rightarrow \infty$.
\newtheorem{rk1}{Remark}
	\begin{rk1}
		We remark that for the case when $\mu=0$ in equation $\eqref{KdVwReg}$, and under the assumption that the solution $w$ belongs to the global attractor of $\eqref{KdV}$, the above estimates are still valid, which yields an explicit bound for $R$ in $\eqref{Rkdv1}$, i.e., $R=\mathcal{R}_{2}|_{\mu=0}=\mathcal{R}_{2}^{0}$.
	\end{rk1}

\subsection{Time derivative bound}
Let $\theta\in \dot{H}^{2}$. Since $\eqref{KdVwReg}$ holds in $L^{2}_{loc}(\mathbb{R}, \dot{H}^{2})$ and $w\in L^{\infty}(\mathcal{R}, \dot{H}^{4})$, we multiply $\eqref{KdVwReg}$ by $\theta$, and integrate over the spatial domain to get
	\begin{align*}
		\int w_{s}\theta +\epsilon \int w_{xxxx}\theta+\int ww_{x}\theta+ \int w_{xxx}\theta+ \gamma \int w\theta+\mu \int Pw\theta=\int (f+\mu v)\theta.
	\end{align*}
Using integration by parts we get,
	\begin{align*}
		\int w_{s}\theta &=-\epsilon \int w_{xx}\theta_{xx}+\frac{1}{2}\int w^{2}\theta_{x}-\int w_{x}\theta_{xx}\\
				        &\quad-\gamma \int w\theta-\mu \int Pw\theta+\int (f+\mu v)\theta\\
				        &\leq \epsilon |w_{xx}||\theta_{xx}|+\frac{1}{2}|w|^{2}|\theta_{x}|_{\infty}+|w_{x}||\theta_{xx}|\\
				        &\quad+ ((\gamma+\mu) |w|+|f|+\mu |v|_{X})|\theta|\\
				        &\leq \{\mathcal{R}_{2}+\frac{1}{2}\mathcal{R}_{0}^{2}+\mathcal{R}_{1}+(\gamma+\mu)\mathcal{R}_{0}+|f|+\mu |v|_{X}\}|\theta|_{H^{2}}.
	\end{align*}
Thus, since $|v|_{X}<\rho$, we have
	\begin{align*}
		|\frac{dw}{ds}(s)|_{\dot{H}^{-2}}\leq \tilde{\mathcal{R'}}, \text{ for all } s\in \mathbb{R},
	\end{align*}	
where
	\begin{align}
		\tilde{\mathcal{R'}}:= \mathcal{R}_{2}+\frac{1}{2}\mathcal{R}_{0}^{2}+\mathcal{R}_{1}+(\gamma+\mu)\mathcal{R}_{0}+|f|+\mu \rho, \label{w'}
	\end{align}
with
	$$\tilde{\mathcal{R'}}=O(\mu^{\frac{14\alpha-1}{12}}),$$
as $\mu \rightarrow \infty$, and $\tilde{\mathcal{R'}}=O(\gamma^{-\frac{10}{3}}, |f|_{H^{2}}^{\frac{7}{3}})$, as $\gamma\rightarrow 0$ and $|f|_{H^{2}}\rightarrow \infty$, after taking $\epsilon >0 $ as small as necessary.

\subsection{Passing to the limit}
Summarizing the above $\epsilon$-independent bounds, we have
	$$w^{\epsilon}\in L^{\infty}(\mathbb{R}, \dot{H}^{2}),$$
 and
 	$$\frac{\partial w^{\epsilon}}{\partial s}\in L^{\infty}(\mathbb{R}, \dot{H}^{-2}),$$
and the bounds for the corresponding norms are uniform in $\epsilon$. Thus, by the Aubin compactness theorem and a Cantor diagonal argument, with respect to the sequence $M$, there exists a subsequence $w^{\epsilon_{j}}$ of $w^{\epsilon}$ and $w\in L^{\infty}(\mathbb{R}, \dot{H}^{2})$ such that
	\begin{align*}
		&w^{\epsilon_{j}}\rightharpoonup w \text{ weak-}\star \text{ in } L^{\infty}(\mathbb{R}, \dot{H}^{2}),\\
		&w^{\epsilon_{j}}\rightarrow w \text{ strongly in } L^{2}_{loc}(\mathbb{R}, \dot{H}^{1}),\\
		&\frac{\partial w^{\epsilon_{j}}}{\partial s}\rightharpoonup \frac{dw}{ds} \text{ weak-}\star \text{ in } L^{\infty}(\mathbb{R}, \dot{H}^{-2}),
	\end{align*}	
as $\epsilon_{j}\rightarrow 0$, and $w$ obeys the uniform in $\epsilon$ bounds we established in previous steps for $w^{\epsilon}$. Thus, we can pass to the limit, in the sense of distributions, to get a solution of $\eqref{KdVw}$ as a limit of one for $\eqref{KdVwReg}$. This completes the proof of the existence of bounded solutions, on all of $\mathbb{R}$, to the equation $\eqref{KdVw}$, i.e., $w\in Y$. We note that since
	\begin{align} \label{eqn3}
		\frac{\partial w}{\partial s}= -ww_{x}-w_{xxx}-\gamma w-\mu Pw+f+\mu v,
	\end{align}
we have that $\frac{\partial w}{\partial s}\in \dot{H}^{-1}$, and the above equation holds in $L^{\infty}(\mathbb{R}, \dot{H}^{-1})$. Thus, similar to the previous section, from $\eqref{eqn3}$, one can show that
	$$\sup_{s\in \mathbb{R}}|\frac{\partial w}{\partial s}(s)|_{\dot{H}^{-1}}\leq \mathcal{R'},$$
where
	$$\mathcal{R'}:= \frac{1}{2}\mathcal{R}_{0}^{\frac{3}{2}}\mathcal{R}_{1}^{\frac{1}{2}}+\mathcal{R}_{2}+(\gamma+\mu)\mathcal{R}_{0}+|f|+\mu \rho.$$
We note that
  	\[ \mathcal{R}^{'} = \begin{cases}
      		O(\mu^{\frac{3\alpha+\beta}{4}}) & \textrm{ if $5\alpha\leq 3\beta+1$,} \\
      		O(\mu^{\frac{14\alpha-1}{12}}) & \textrm{ if $5\alpha> 3\beta+1$,} \\
   	\end{cases} \]

\section{Uniqueness of the solution}

\newtheorem{ANA}{Theorem}
\begin{ANA}
\label{uniq1}
Let $\rho=4R$, where $R$ is given in $\eqref{Rkdv1}$. Assume that conditions $\eqref{condition1}$, $\eqref{condition2}$ hold along with
	\begin{align}
		C_{3}\leq 2\mu, \label{condition3} 
	\end{align}
and assume that $m$ is large enough such that 	
	\begin{align}
		 \frac{C_{3}L^{2}}{8\pi^{2}(m+1)^{2}}\frac{1}{\gamma^{2}}\left[ (2\gamma+2\mu)\mathcal{R}_{\infty}+2\mathcal{R^{'}}^{4}\gamma^{-3}\right]\	
\leq \frac{1}{2}, \label{condition4} 
	\end{align}	
holds where $C_{3}$ is defined in $\eqref{c3}$, below. Then for any $v\in \mathcal{B}_{X}^{\rho}(0)\subset X$, where X is defined as in $\eqref{X}$, there exists a unique bounded solution $w\in Y$ of $\eqref{KdVw}$, where $Y$ is defined as in $\eqref{Y}$.
\end{ANA}

\begin{proof}
Suppose there exist two bounded solutions of $\eqref{KdVw}$, $w$ and $\tilde{w}$, in $Y$ corresponding to the same $v\in \mathcal{B}_{X}^{\rho}(0)$. 
	$$w_s+ww_{x}+w_{xxx}+\gamma w=f- \mu [P_m(w)-v],$$
	$$\tilde{w}_s+\tilde{w}\tilde{w}_{x}+\tilde{w}_{xxx}+\gamma \tilde{w}=f- \mu [P_m(\tilde{w})-v].$$
Subtract, denoting $\delta:=w-\tilde{w}$, to obtain
	\begin{align*}
		\delta_s+ww_{x}-\tilde{w}\tilde{w}_{x} +\delta_{xxx}+\gamma \delta=-\mu P_m\delta \label{deltaeta3}.
	\end{align*}
Note that
	\begin{align*}
		ww_{x}-\tilde{w}\tilde{w}_{x}=\frac{1}{2}\frac{\partial }{\partial x}(w^{2}-\tilde{w}^{2})=(\xi \delta)_{x},	\qquad \text{ where } \xi=\frac{w+\tilde{w}}{2},			
	\end{align*}\\\
and hence
	\begin{equation}
		\delta_s+(\xi \delta)_{x}+\delta_{xxx}+ \gamma \delta=-\mu P \delta. \label{attractordelta0}
	\end{equation}\\\
Note that the equation for the difference $\eqref{attractordelta0}$ holds only in $L^{\infty}(\mathbb{R}, \dot{H}^{-1})$. Thus, the equation does not act on $\delta_{xx}\in \dot{L}^{2}$. But it acts on $e^{ikx\frac{2\pi}{L}}$(for simplicity, we take $L=2\pi$ for this section). So we can write $\eqref{attractordelta0}$ at the level of $k$-th Fourier coefficient,
	\begin{align}
		\frac{d\delta_k}{ds}+ikc_{k}-ik^{3}\delta_{k}+ \gamma \delta_{k}= -\mu \delta_{k}\chi_{|k|\leq m}, \label{Fourier}
	\end{align}
where $k\in \mathbb{Z}\setminus\{0\}$, $\delta_{k}$ is the $k$-th Fourier coefficient of $\delta$, $c_{k}$ is the $k$-th Fourier coefficient of $\xi \delta\in L^{\infty}(\mathbb{R}, H^{2})$, and $\chi_{k\leq m}$ is $1$ when $|k|\leq m$ and zero otherwise. $\eqref{Fourier}$ is an ordinary differential equation in $\mathbb{C}$. We multiply $\eqref{Fourier}$ by $k^{2}\bar{\delta}_{k}\in \mathbb{C}$, to get
	\begin{align*}
		k^{2}\frac{d\delta_k}{ds}\bar{\delta}_{k}+ik^{3}c_{k}\bar{\delta}_{k}-ik^{5}|\delta_{k}|^{2}+ \gamma k^{2}|\delta_{k}|^{2}= -\mu k^{2}|\delta_{k}|^{2}\chi_{|k|\leq m}.
	\end{align*}
We take the real parts above to obtain
	\begin{align}
		\frac{1}{2}k^{2}\frac{d|\delta_k|^{2}}{ds}+Re\{ik^{3}c_{k}\bar{\delta}_{k}\}+ \gamma k^{2}|\delta_{k}|^{2}= -\mu k^{2}|\delta_{k}|^{2}\chi_{|k|\leq m}. \label{dom}
	\end{align}
We claim that
	\begin{align}
		\sum_{k\in \mathbb{Z}\setminus\{0\}}k^{2}\frac{d|\delta_k|^{2}}{ds}= \frac{d}{ds} \sum_{k\in \mathbb{Z}\setminus\{0\}}k^{2}|\delta_k|^{2}. \label{chder}
	\end{align}
To prove the above claim, we apply a consequence of the Weirstrass convergence theorem, i.e., the series of derivatives converges uniformly. Therefore, we need to find a sequence $g(k)\in \ell^{1}$ such that $|k^{2}\frac{d|\delta_k|^{2}}{ds}|\leq g(k)$. From $\eqref{dom}$, we have
	\begin{align*}
		|\frac{1}{2}k^{2}\frac{d|\delta_k|^{2}}{ds}|&=|-Re\{ik^{3}c_{k}\bar{\delta}_{k}\}-\gamma k^{2}|\delta_{k}|^{2}-\mu k^{2}|\delta_{k}|^{2}\chi_{|k|\leq m}|\\
			&\leq k^{3}|c_{k}||\delta_{k}|+(\gamma+\mu)k^{2}|\delta_{k}|^{2}\\
			&\leq \frac{1}{2}k^{2}|c_{k}|^{2}+\frac{1}{2}k^{4}|\delta_{k}|^{2}+(\gamma+\mu)k^{2}|\delta_{k}|^{2}.
	\end{align*}
Thus,
	\begin{align*}
		|k^{2}\frac{d|\delta_k|^{2}}{ds}|\leq k^{2}|c_{k}|^{2}+k^{4}|\delta_{k}|^{2}+2(\gamma+\mu)k^{2}|\delta_{k}|^{2}=:g(k).
	\end{align*}
Then,
	\begin{align*}
		\sum_{k\in \mathbb{Z}\setminus\{0\}} g(k)&=\sum_{k\in \mathbb{Z}\setminus\{0\}}k^{2}|c_{k}|^{2}+\sum_{k\in \mathbb{Z}\setminus\{0\}}k^{4}|\delta_{k}|^{2}+2(\gamma+\mu)\sum_{k\in \mathbb{Z}\setminus\{0\}}k^{2}|\delta_{k}|^{2}\\
		&=|(\xi \delta)_{x}|^{2}+|\delta_{xx}|^{2}+2(\gamma+\mu)|\delta_{x}|^{2}\\
		&\leq 16\mathcal{R}_{0}^{2}\mathcal{R}_{1}^{2}+4\mathcal{R}_{2}^{2}+8(\gamma+\mu)\mathcal{R}_{1}^{2}< \infty,
	\end{align*}		
thus $g(k)\in \ell^{1}$, which concludes the proof of $\eqref{chder}$. Thus we can take sum over all $k\in \mathbb{Z}\setminus\{0\}$ in $\eqref{dom}$, to get
	\begin{align}
		\frac{1}{2}\frac{d}{ds}\sum_{k\in \mathbb{Z}\setminus\{0\}}k^{2}|\delta_k|^{2}&+\sum_{k\in \mathbb{Z}\setminus\{0\}}Re\{ik^{3}c_{k}\bar{\delta}_{k}\}+ \gamma \sum_{k\in \mathbb{Z}\setminus\{0\}}k^{2}|\delta_{k}|^{2}\notag\\
		&= -\mu \sum_{k\in \mathbb{Z}\setminus\{0\}}k^{2}|\delta_{k}|^{2}\chi_{|k|\leq m}. \label{sumfourier}
	\end{align}
We realize that
	\begin{align*}
		\sum_{k\in \mathbb{Z}\setminus\{0\}}Re\{ik^{3}c_{k}\bar{\delta}_{k}\}&=\sum_{k\in \mathbb{Z}\setminus\{0\}}Re\{-ik^{3}\delta_{k}\bar{c}_{k}\}\\
		&=Re\{\sum_{k\in \mathbb{Z}\setminus\{0\}}-ik^{3}\delta_{k}\bar{c}_{k}\}\\
		&=\int \delta_{x}(\xi \delta)_{xx}.
	\end{align*}
Thus $\eqref{sumfourier}$ becomes
	\begin{align}
		\frac{1}{2}\frac{d}{ds}|\delta_{x}|^{2}+\int \delta_{x}(\xi \delta)_{xx} + \gamma |\delta_{x}|^{2}= -\mu |P\delta_{x}|^{2}. \label{hisse1}
	\end{align}
Take $\dot{H}^{-2}$ action of $\eqref{attractordelta0}$ on $-\xi \delta \in H^{2}$ to obtain
	\begin{align}
		-<\delta_{s},\xi \delta>-\int \delta_{x}(\xi \delta)_{xx} - \gamma \int \xi \delta^{2}= \mu \int \xi \delta P\delta. \label{hisse2}
	\end{align}
One can verify, since $H^{1}$ is an algebra, that
	\begin{align*}
		\frac{d}{ds}\int \xi \delta^{2}= <\xi_{s}, \delta^{2}>+ 2<\delta_{s}, \xi\delta>.
	\end{align*}
Therefore, we use this observation and add $\eqref{hisse1}$ and $\eqref{hisse2}$, to get
	\begin{align*}
		\frac{d}{ds}\Psi+\gamma \Psi= -\gamma \int \delta_{x}^{2}+\gamma \int \xi \delta^{2}-2\mu \int (P\delta_{x})^{2}+2\mu \int \xi \delta P\delta-<\xi_{s},\delta^{2}>,
 	\end{align*}
where
	\begin{align}
		\Psi(\delta(s))=\int (\delta_{x}^{2}(s)-\xi(s)\delta^{2}(s)). \label{F3} 
 	\end{align}
Here we do the following estimations on each of these terms using H\"older, Young's and Agmon's inequalities:
	$$\gamma \int \xi \delta^{2}\leq \gamma |\xi|_{\infty}|\delta|^{2}\leq \gamma \mathcal{R}_{\infty} |\delta|^{2},$$
	$$-2\mu \int \xi \delta P\delta \leq 2\mu\mathcal{R}_{\infty} |\delta|^{2},$$
	$$<\xi_{s},\delta^{2}>\leq |\xi_{s}|_{H^{-1}}|\delta^{2}|_{H^{1}}\leq 2\mathcal{R}^{'}|\delta_{x}|^{\frac{3}{2}}|\delta|^{\frac{1}{2}}\leq 2\mathcal{R^{'}}^{4}\gamma^{-3}|\delta|^{2}+\gamma |\delta_{x}|^{2}.$$
Thus we get
	\begin{align*}
		\frac{d}{ds}\Psi+\gamma \Psi&\leq \left[ (\gamma+2\mu)\mathcal{R}_{\infty}+2\mathcal{R^{'}}^{4}\gamma^{-3}\right]|\delta(s)|^{2}\\
		&\leq \left[ (\gamma+2\mu)\mathcal{R}_{\infty}+2\mathcal{R^{'}}^{4}\gamma^{-3}\right]\sup_{s\in \mathbb{R}}|\delta(s)|^{2}.
 	\end{align*}	
Let $s_{0}\in \mathbb{R}$, and $s>s_{0}$. From Gronwall's inequality,
	\begin{align*}
		\Psi(\delta(s))\leq \Psi(\delta(s_{0}))e^{-\gamma(s-s_{0})}+\frac{1}{\gamma}\left[ (\gamma+2\mu)\mathcal{R}_{\infty}+2\mathcal{R^{'}}^{4}\gamma^{-3}\right]\sup_{s\in \mathbb{R}}|\delta(s)|^{2}.
 	\end{align*}	
Since
	\begin{align*}
		\Psi(\delta(s))\geq |\delta_{x}(s)|^{2}-|\xi|_{\infty}|\delta(s)|^{2}\geq |\delta_{x}(s)|^{2}-\mathcal{R}_{\infty}|\delta(s)|^{2},		
	\end{align*}
we have that
	\begin{align}
		|\delta_{x}(s)|^{2}\leq \Psi(\delta(s))+\mathcal{R}_{\infty}|\delta(s)|^{2}. \label{eqn4}
	\end{align}	
Thus,
	\begin{align}
		|\delta_{x}(s)|^{2}\leq \Psi(\delta(s_{0}))e^{-\gamma(s-s_{0})}+\frac{1}{\gamma}\left[ (2\gamma+2\mu)\mathcal{R}_{\infty}+2\mathcal{R^{'}}^{4}\gamma^{-3}\right]\sup_{s\in \mathbb{R}}|\delta(s)|^{2}.
	\end{align}
Since $\Psi(\delta(s_0))$ is uniformly bounded for all $s_{0}\in\mathbb{R}$, we let $s_{0}\rightarrow -\infty$, to obtain a `reverse' Poincar\'e type inequality
	\begin{align}
		|\delta_{x}(s)|^{2}\leq \frac{1}{\gamma}\left[ (2\gamma+2\mu)\mathcal{R}_{\infty}+2\mathcal{R^{'}}^{4}\gamma^{-3}\right]\sup_{s\in \mathbb{R}}|\delta(s)|^{2}, \label{revP}
	\end{align}
for every $s\in \mathbb{R}$. We now take the action of $\eqref{attractordelta0}$ on $2\delta$, and observe that
	\begin{align*}
		<\delta_{s}, \delta>=\frac{1}{2}\frac{d}{ds}|\delta|^{2},
	\end{align*}
and apply $\eqref{hoca1}$ to obtain
	\begin{align*}
		\frac{d}{ds}|\delta|^{2}+2\gamma |\delta|^{2}+2\mu |P\delta|^{2}=-\int 2(\xi \delta)_{x}\delta.
	\end{align*}
Here
	\begin{align} \label{eqn66}
		-\int 2(\xi \delta)_{x}\delta= \int 2\xi \delta \delta_{x}&= \int \xi \frac{\partial}{\partial x}\delta^{2}=-\int \xi_{x}\delta^{2}\notag\\
											      &\leq |\xi_{x}|_{\infty}|\delta|^{2}\notag\\
											      &\leq |\xi_{x}|^{\frac{1}{2}}|\xi_{xx}|^{\frac{1}{2}}|\delta|^2\notag\\
											      &\leq \mathcal{R}_{1}^{\frac{1}{2}}\mathcal{R}_{2}^{\frac{1}{2}}|\delta|^2\notag\\
											      &=C_{3}|\delta|^{2},
	\end{align}					
where
	\[ C_{3} = \begin{cases}
      		O(\mu^{\frac{3\alpha+3\beta-2}{8}}) & \textrm{ if $5\alpha\leq 3\beta+1$,} \\
      		O(\mu^{\frac{24\alpha-9}{24}}) & \textrm{ if $5\alpha> 3\beta+1$,} \\
   	\end{cases} \]
for $\alpha\in [1,2)$. For specific choices of $\alpha=1$ and $\beta=\frac{4}{3}$, we have	
	\begin{align}
		C_{3}:=  \mathcal{R}_{1}^{\frac{1}{2}}\mathcal{R}_{2}^{\frac{1}{2}}=O(\mu^{\frac{5}{8}}), \label{c3}
	\end{align}
as $\mu \rightarrow \infty$. Thus, 	
	\begin{align*}						
		-\int 2(\xi \delta)_{x}\delta&\leq C_{3}|P\delta|^{2}+C_{3}|Q\delta|^{2}\\
						       &\leq C_{3}|P\delta|^{2}+ \frac{C_{3}L^{2}}{4\pi^{2}(m+1)^{2}}|\delta_{x}|^{2}.\\
         \end{align*}
If we choose $\mu$ large enough such that $\eqref{condition3}$ hold, then
	\begin{align*}
		\frac{d}{ds}|\delta|^{2}+2\gamma |\delta|^{2}\leq  \frac{C_{3}L^{2}}{4\pi^{2}(m+1)^{2}}|\delta_{x}|^{2}.
	\end{align*}
Applying $\eqref{revP}$ above, we get 	
	\begin{align*}
		\frac{d}{ds}|\delta|^{2}+2\gamma |\delta|^{2}\leq  \frac{C_{3}L^{2}}{4\pi^{2}(m+1)^{2}}\frac{1}{\gamma}\left[ (2\gamma+2\mu)\mathcal{R}_{\infty}+2\mathcal{R^{'}}^{4}\gamma^{-3}\right]\sup_{s\in \mathbb{R}}|\delta(s)|^{2}.
	\end{align*}	
As before, let $s_{0}\in \mathbb{R}$, and $s>s_{0}$. Since $|\delta(s_{0})|$ is bounded uniformly for all $s_{0}\in \mathbb{R}$, we apply Gronwall's inequality on the interval $[s_{0}, s]$, and let $s_{0}\rightarrow -\infty$, to obtain
	\begin{align*}
		|\delta(s)|^{2}\leq \frac{C_{3}L^{2}}{8\pi^{2}(m+1)^{2}}\frac{1}{\gamma^{2}}\left[ (2\gamma+2\mu)\mathcal{R}_{\infty}+2\mathcal{R^{'}}^{4}\gamma^{-3}\right]\sup_{s\in \mathbb{R}}|\delta(s)|^{2},
	\end{align*}
for all $s\in \mathbb{R}$. By choosing $m$ large enough such that $\eqref{condition4}$ holds, we obtain that
	\begin{align*}
		\sup_{s\in \mathbb{R}}|\delta(s)|^{2}\leq \frac{1}{2}\sup_{s\in \mathbb{R}}|\delta(s)|^{2}.
	\end{align*}
Thus, $\delta(s)=0$, for all $s\in \mathbb{R}$, i.e., $w=\tilde{w}$.
\end{proof}

\newtheorem{rk15}[rk1]{Remark}
	\begin{rk15}
		\begin{enumerate}
			\item 	Note that conditions $\eqref{condition1}$ and $\eqref{condition2}$ enable us to choose $\mu$ large enough such that $\eqref{condition3}$ holds.
			\item 	By Theorem \ref{uniq1}, we can now define a map $W: \mathcal{B}_{X}^{\rho}(0)\rightarrow Y$ such that $W(v):=w$ is the unique bounded solution of $\eqref{KdVw}$, i.e., $w\in Y$.
		\end{enumerate}
	\end{rk15}

\newtheorem{rk2}[rk1]{Remark}
	\begin{rk2}
		\begin{enumerate}
			\item 	$\eqref{condition1}$, $\eqref{condition2}$, and $\eqref{condition4}$ dictates that
			\[ d = \begin{cases}
      		\max\{\frac{4-\alpha}{2}, \frac{7\alpha-3\beta+4}{6}, \frac{27\alpha+11\beta-2}{16}\}=\frac{27\alpha+11\beta-2}{16} & \textrm{ if $5\alpha\leq 3\beta+1$,} \\
      		\max\{\frac{4-\alpha}{2}, \frac{7\alpha-3\beta+4}{6}, \frac{136\alpha-17}{48}\}=\frac{136\alpha-17}{48} & \textrm{ if $5\alpha> 3\beta+1$,} \\
   	\end{cases} \]
where $m=O(\mu^{d})$, as $\mu \rightarrow \infty$. $\eqref{condition3}$ dictates that
	\[  \begin{cases}
      		3\alpha+3\beta<10 & \textrm{ if $5\alpha\leq 3\beta+1$,} \\
      		\alpha<\frac{11}{8} & \textrm{ if $5\alpha> 3\beta+1$.} \\
   	\end{cases} \]
We recall that $\alpha\in [1, 2)$. Thus, this is a linear optimization problem where we would like to minimize $d$ subject to the constraints. Solving the linear optimization problem, we obtain that the optimal values are $\alpha=1$ and $\beta=\frac{4}{3}$ where the minimum value of $d$ is $\frac{119}{48}$. Thus, the minimum number of the modes, $m$, that is needed is to achieve the uniqueness of the bounded solution of $\eqref{KdVw}$ is $O(\mu^{d})$, where $d= \frac{119}{48}\approx 2.48$
			\item 	For the optimal values of $\alpha$ and $\beta$, we realize that $\mathcal{R}_{0}=O(1),  \mathcal{R}_{1}= O(\mu^{\frac{1}{6}}), \mathcal{R}_{2}=O(\mu^{\frac{13}{12}}), \mathcal{R'}=O(\mu^{\frac{13}{12}}),$ and $\mathcal{R}_{\infty}= O(\mu^{\frac{1}{12}})$, as $\mu \rightarrow \infty$.
		\end{enumerate}
	\end{rk2}
	
\section{Lipschitz property of $W(v)$}

\newtheorem{W2}[ANA]{Theorem}
\begin{W2}
\label{LP}
Let $v\in \mathcal{B}_{X}^{\rho}(0):=\{v\in X; |v|_X< \rho \}$ with $\rho=4R$, and $R$ is given in $\eqref{Rkdv1}$. Assume conditions $\eqref{condition1}$, $\eqref{condition2}$, $\eqref{condition3}$ and $\eqref{condition4}$ hold. Then the map $P_mW: \mathcal{B}_{X}^{\rho}(0)\rightarrow X$ is a Lipschitz function with Lipschitz constant $L_{W}$, where
	\begin{align}
		L_{W}=\frac{4\pi^{2}m^{2}}{L^{2}}\left( \frac{C_{3}L^{2}}{2\pi^{2}\gamma(m+1)^{2}}(\mu+\mu  \mathcal{R}_{\infty})+\frac{2\mu}{\gamma}\right). \label{LW2}
	\end{align}
\end{W2}

\begin{proof}
Note that all constants $\mathcal{R}_0, \mathcal{R}_1, \mathcal{R}_2, \mathcal{R'}$ and $\mathcal{R}_\infty$ depend on $\rho=4R$. Let $v, \tilde{v}\in \mathcal{B}_{X}^{\rho}(0)$, with $W(v)=w$ and $W(\tilde{v})=\tilde{w}$, so that
	$$w_s+ww_{x}+w_{xxx}+\gamma w=f- \mu [P_m(w)-v],$$
	$$\tilde{w}_s+\tilde{w}\tilde{w}_{x}+\tilde{w}_{xxx}+\gamma \tilde{w}=f- \mu [P_m(\tilde{w})-\tilde{v}].$$
Subtract, denoting $\delta:=w-\tilde{w}$ and $\eta:=v-\tilde{v}$, to obtain
	\begin{align}
		\delta_s+(\xi\delta)_{x}+\delta_{xxx}+\gamma \delta+\mu P_m\delta= \mu \eta \label{deltaeta3},
	\end{align}
where $\xi=\frac{w+\tilde{w}}{2}$. As in the previous section, we consider the evolution of the Fourier coefficients of $\delta$ to verify that
	\begin{align}
		\frac{d}{ds}\Psi+\gamma \Psi=&-\gamma \int \delta_{x}^{2}+\gamma \int \xi \delta^{2}-2\mu \int (P\delta_{x})^{2}+2\mu \int \xi \delta P\delta-<\xi_{s},\delta^{2}> \notag \\
						                &2\mu\int \eta_{xx}\delta-2\mu\int \xi \eta \delta, \label{func3EnergyEta}
 	\end{align}
where
	\begin{align*}
		\Psi(\delta)=\int (\delta_{x}^{2}-\xi\delta^{2}),
 	\end{align*}
as in $\eqref{F3}$. For last three terms on the right-hand side of $\eqref{func3EnergyEta}$, we have
	\begin{align*}
		&<\xi_{s},\delta^{2}>\leq |\xi_{s}|_{H^{-1}}|\delta^{2}|_{H^{1}}\leq 2\mathcal{R}^{'}|\delta_{x}|^{\frac{3}{2}}|\delta|^{\frac{1}{2}}\leq 2\mathcal{R^{'}}^{4}\gamma^{-3}|\delta|^{2}+\gamma |\delta_{x}|^{2}.\\
		&2\mu\int \eta_{xx}\delta\leq 2\mu |\eta|_{X}|\delta|, \\
		&-2\mu\int \xi \eta \delta \leq 2\mu |\xi|_{\infty}|\eta|_{X}|\delta|\leq 2\mu \mathcal{R}_{\infty}|\eta|_{X}|\delta|.
 	\end{align*}
Thus we get
	\begin{align*}
		\frac{d}{ds}\Psi(\delta(s))+\gamma \Psi(\delta(s))\leq \left[ (\gamma+2\mu)\mathcal{R}_{\infty}+2\mathcal{R^{'}}^{4}\gamma^{-3}\right]|\delta(s)|^{2}+ |\eta|_{X}(2\mu+2\mu  \mathcal{R}_{\infty})|\delta(s)|.
 	\end{align*}
Let $s_{0}\in \mathbb{R}$, and $s>s_{0}$. Since $\Psi(\delta(s_{0}))$ is uniformly bounded for all $s_{0}\in \mathbb{R}$, we apply Gronwall's inequality on the interval $[s_{0}, s]$, use $\eqref{eqn4}$ and let $s_{0}\rightarrow -\infty$, to obtain
	\begin{align}
		|\delta_{x}(s)|^{2}\leq \frac{1}{\gamma}[ (2\gamma+2\mu)\mathcal{R}_{\infty}+&2\mathcal{R^{'}}^{4}\gamma^{-3}]\sup_{s\in \mathbb{R}}|\delta(s)|^{2}\notag \\
		&+|\eta|_{X}(2\mu+2\mu  \mathcal{R}_{\infty})\sup_{s\in \mathbb{R}}|\delta(s)|, \label{revP2}
	\end{align}
for all $s\in \mathbb{R}$. Now we take the action of equation $\eqref{deltaeta3}$ on $2\delta$, to obtain
	\begin{align*}
		\frac{d}{ds}|\delta|^{2}+2\gamma |\delta|^{2}+2\mu |P\delta|^{2}=-\int 2(\xi \delta)_{x}\delta+2\mu \int \eta \delta.
	\end{align*}
Thanks to $\eqref{eqn66}$, we have
	\begin{align*}
		-\int 2(\xi \delta)_{x}\delta       &\leq C_{3}|\delta|^{2},\\
							     &= C_{3}|P\delta|^{2}+C_{3}|Q\delta|^{2}\\
						       	     &\leq C_{3}|P\delta|^{2}+ \frac{C_{3}L^{2}}{4\pi^{2}(m+1)^{2}}|\delta_{x}|^{2}.
	\end{align*} 	
Also, we have
	\begin{align*}
		2\mu \int \eta \delta\leq 2\mu |\eta||\delta|\leq 2\mu|\eta|_{X}|\delta|.
	\end{align*}	
Thus, from $\eqref{condition3}$ and the above estimates, we get
	\begin{align*}
		\frac{d}{ds}|\delta|^{2}+2\gamma |\delta|^{2}\leq \frac{C_{3}L^{2}}{4\pi^{2}(m+1)^{2}}|\delta_{x}|^{2}+ 2\mu|\eta|_{X}|\delta|.
	\end{align*}
We use $\eqref{revP2}$ above, to get	
	\begin{align*}
		\frac{d}{ds}|\delta|^{2}+2\gamma |\delta|^{2}\leq & \frac{C_{3}L^{2}}{4\pi^{2}(m+1)^{2}}\frac{1}{\gamma}\left[ (2\gamma+2\mu)\mathcal{R}_{\infty}+2\mathcal{R^{'}}^{4}\gamma^{-3}\right]\sup_{s\in \mathbb{R}}|\delta(s)|^{2}\\
		&+ \frac{C_{3}L^{2}}{4\pi^{2}(m+1)^{2}}|\eta|_{X}(2\mu+2\mu  \mathcal{R}_{\infty})\sup_{s\in \mathbb{R}}|\delta(s)|\\
		&+ 2\mu|\eta|_{X}\sup_{s\in \mathbb{R}}|\delta(s)|.
	\end{align*}	
By $\eqref{condition4}$, we have
	\begin{align*}
		\frac{d}{ds}|\delta(s)|^{2}+2\gamma |\delta(s)|^{2}&\leq \gamma\sup_{s\in \mathbb{R}}|\delta(s)|^{2}+\gamma \frac{L^{2}}{4\pi^{2}m^{2}}L_{W}|\eta|_{X}\sup_{s\in \mathbb{R}}|\delta(s)|,
	\end{align*}	
where $L_{W}$ is defined in $\eqref{LW2}$. We let $s_{0}\in \mathbb{R}$, and $s>s_{0}$. Since $|\delta(s_{0})|^{2}$ is bounded for all $s_{0}\in \mathbb{R}$, by Gronwall's inequality on the interval $[s_{0}, s]$, and letting $s_{0}\rightarrow -\infty$, we obtain
	\begin{align*}
		|\delta(s)|^{2}\leq \frac{1}{2}\sup_{s\in \mathbb{R}}|\delta(s)|^{2}+ \frac{1}{2}\frac{L^{2}}{4\pi^{2}m^{2}}L_{W}|\eta|_{X}\sup_{s\in \mathbb{R}}|\delta(s)|,
	\end{align*}
for all $s\in \mathbb{R}$. Thus
	\begin{align*}
		\sup_{s\in \mathbb{R}}|\delta(s)|\leq \frac{L^{2}}{4\pi^{2}m^{2}}L_{W} |\eta|_{X}.
	\end{align*}
We note that
	\begin{align*}
		|P\delta_{xx}(s)|\leq \frac{4\pi^{2}m^{2}}{L^{2}}|P\delta(s)|\leq  \frac{4\pi^{2}m^{2}}{L^{2}}|\delta(s)|\leq L_{W}|\eta|_{X}.
	\end{align*}
Thus,
	$$|P\delta |_X \leq L_{W} |\eta|_X,$$
i.e.,
	$$|PW(v)-PW(\tilde{v})|_X\leq L_W |v-\tilde{v}|_X,$$
where $L_{W}$ is defined in $\eqref{LW2}$.
\end{proof}

\newtheorem{w=u2}[ANA]{Theorem}
\begin{w=u2}
\label{DMW}
Let $\rho=4R$, where $R$ is given by $\eqref{Rkdv1}$. Assume that $\mu$ and $m$ are large enough such that $\eqref{condition1}$, $\eqref{condition2}$, $\eqref{condition3'}$ and $\eqref{condition4}$ hold.
 \begin{enumerate}
 \item Suppose that $u$ is a trajectory in the global attractor of the damped and driven KdV equation $\eqref{KdV}$-$\eqref{boundary}$, and suppose that $w=W(P_{m}u)$ is the unique bounded solution of equation $\eqref{KdVw}$ with $v=P_mu$.  Then $w(s)=W\left(P_{m}u\right)(s)=u(s)$, for all $s\in \mathbb{R}$.
 \item Suppose $v=P_m W(v)$, for some $v\in \mathcal{B}_X^\rho(0)$. Then $W(v)$ is a trajectory in the global attractor of the damped and driven KdV equation $\eqref{KdV}$-$\eqref{boundary}$.
\end{enumerate}
\end{w=u2}	

\begin{proof}
To prove (1) we first observe that the choice of $R$ in $\eqref{Rkdv1}$ guarantees that $P_{m}u(\cdot)\in \mathcal{B}_{X}^{\rho}(0)$, for every $u(\cdot)\subset \mathcal{A}$. Taking the difference of the following equations

	\begin{align*}
		&w_s+ww_{x}+w_{xxx}+\gamma w=f- \mu [P(w-u)], \\
		&u_s+uu_{x}+u_{xxx}+\gamma u=f,
	\end{align*}
we get
	\begin{align*}
		\delta_s+ww_{x}-uu_{x}+\delta_{xxx}+\gamma \delta= -\mu P\delta,
	\end{align*}\\\
where $\delta:=w-u$. Note that
	\begin{align*}
		ww_{x}-uu_{x}&=(\xi \delta)_{x},	\qquad \text{ where } \xi=\frac{w+u}{2},			
	\end{align*}\\\
and hence
	\begin{equation}
		\delta_s+(\xi \delta)_{x}	+\delta_{xxx}+ \gamma \delta= -\mu P \delta.
	\end{equation}\\\
We then proceed as in the proof of Theorem \ref{uniq1} to obtain that $|\delta(s)|=0$, for all fixed $s\in \mathbb{R}$. Thus, $w(s)=W(P_{m}u(s))=u(s)$, for all $s\in \mathbb{R}$; which concludes the proof of (1).

To prove (2) we observe that since $v = P_mW(v)$, then the unique bounded solution, for all $s\in \mathbb{R}$,   $w(s)=W(v)(s)$, of equation \eqref{KdVw} is in fact a bounded solution, for all $s\in \mathbb{R}$, of the damped and driven KdV equation \eqref{KdV}. Therefore, by the characterization of the global attractor to be the collection of all bounded, for all $s\in \mathbb{R}$, solutions we conclude that  $w(\cdot) \subset \mathcal{A}$.
 	
\end{proof}

\section{The determining form}
Thanks to $\eqref{Rkdv1}$, we have
	\begin{align*}
		|u(s)|_{H^{2}}\leq R, \text{ for all } s\in\mathbb{R},
	\end{align*}
for every $u(\cdot)\subset \mathcal{A}$. In particular, we also have that $|P_{m}u(s)|_{H^{2}}\leq R$ for all $s\in\mathbb{R}$. Let $u^*$ be a steady state solution of the damped and driven KdV $\eqref{KdV}$. As in \cite{Form2}, we consider the following  \emph{determining form} for the damped and driven KdV in the ball $\mathcal{B}_{X}^{\rho}(0)$, with $\rho=4R$,
	\begin{align}
		\frac{dv(\tau)(s)}{d\tau}= -|v(\tau)-P_mW(v(\tau))|_{X}^2 (v(\tau)(s)- P_mu^*), \label{newformkdv}
	\end{align}
	$$v(0)=v_{0},$$
with $v_{0}\in \mathcal{B}_{X}^{\rho}(0)$, where $|\cdot|_{X}$ is defined in $\eqref{X}$. The specific conditions on $m$, and its dependence on $R$, guarantee the existence of a Lipschitz map $P_mW(v)$, for $v\in \mathcal{B}_{X}^{\rho}(0)$, which have been already proven and are stated again in the following:
\newtheorem{Formthmkdv}[ANA]{Theorem}
\begin{Formthmkdv}
\label{DF}
Let $\rho=4R$, where $R$ is given in $\eqref{Rkdv1}$, and suppose that conditions $\eqref{condition1}$,  $\eqref{condition2}$,  $\eqref{condition3}$, and $\eqref{condition4}$ hold.
	\begin{enumerate}
		\item{The vector field in the determining form $\eqref{newformkdv}$ is a Lipschitz map from the ball $\mathcal{B}_X^\rho(0)$ into $X$. Thus, $\eqref{newformkdv}$, with initial data $v_{0}\in \mathcal{B}_{X}^{\rho}(0)$, is an ODE, in $\mathcal{B}_{X}^{\rho}(0)$, which has short time existence and uniqueness. }
		\item{The ball $\mathcal{B}_X^{3R}(P_mu^*)= \{v\in X: |v-P_mu^*|_X< 3R\} \subset \mathcal{B}_X^\rho(0)$ is forward invariant in time, under the dynamics of the determining form $\eqref{newformkdv}$. Consequently, $\eqref{newformkdv}$ has global existence and uniqueness for all initial data in $\mathcal{B}_X^{3R}(P_mu^*)$.}
		\item{Every solution of  the determining form $\eqref{newformkdv}$, with initial data $v_{0}\in\mathcal{B}_X^{3R}(P_mu^*)$, converges to a steady state solution of the determining form $\eqref{newformkdv}$. }
		\item{All the steady state solutions of the determining form, $\eqref{newformkdv}$, that are contained in the ball $\mathcal{B}_X^\rho(0)$ are given by the form $v(s)=P_mu(s)$, for all $s\in \mathbb{R}$, where $u(s)$ is a trajectory that lies on the global attractor, $\mathcal{A}$, of $\eqref{KdV}$-$\eqref{boundary}$.}
	\end{enumerate}
\end{Formthmkdv}
\begin{proof}
    \begin{enumerate}
    	\item Since $P_mW(v)$ is a Lipschitz map in $\mathcal{B}_X^\rho(0)$, which is shown in the previous section, the vector field in the determining form is a Lipschitz map from the ball $\mathcal{B}_X^\rho(0)$ into $X$.
        	\item From the short time existence and uniqueness of determining form $\eqref{newformkdv}$, we find that
        		\begin{align}
        			v(\tau)-P_{m}u^{*}=\theta(\tau)(v_{0}-P_{m}u^{*}), \label{jj1}
        		\end{align}
        		where
        		\begin{align}
        			\theta(\tau)=e^{-\int_{0}^{\tau}|v(\sigma)-P_{m}W(v(\sigma))|_{X}^{2}d\sigma}, \label{jj2}
        		\end{align}
        		for small enough $\tau$. Since $\theta(\tau)$ is non-increasing and belongs to $[0, 1]$, then $v(\tau)\in \mathcal{B}_X^{3R}(P_mu^*)$, which implies the global existence for $\eqref{newformkdv}$. Moreover, $\eqref{jj1}$ implies that $v(\tau)\in \mathcal{B}_X^{3R}(P_mu^*)$, for all $\tau\geq 0$.
	\item Since $\theta(\tau)\in [0, 1]$, for all $\tau\geq 0$, and non-increasing, then $\theta(\tau)\rightarrow \bar{\theta}$, as $\tau \rightarrow \infty$.  On the one hand, if $\bar{\theta}=0$, from $\eqref{jj1}$, 		we obtain that
        		\begin{align*}
        			v(\tau)\rightarrow P_{m}u^{*},
        		\end{align*}
          	 as $\tau \rightarrow \infty$. Thus $v(\tau)$ is converging to steady state solution $P_{m}u^{*}$ of the determining form $\eqref{newformkdv}$, and $P_{m}u^{*}\in P_{m}\mathcal{A}$. On the other hand, if $\bar{\theta}\in (0,1]$, then
		\begin{align*}
			v(\tau)\rightarrow \bar{v}:=\bar{\theta}v_{0}+(1-\bar{\theta})P_{m}u^{*}.
		\end{align*}
		Thus,
		\begin{align}
			|v(\tau)-P_{m}W(v(\tau))|_{X}^{2}\rightarrow |\bar{v}-P_{m}W(\bar{v})|_{X}^{2}. \label{jj3}
		\end{align}
		Now since $\bar{\theta}>0$, from $\eqref{jj2}$, we have that
		\begin{align}
			\int_{0}^{\infty}|v(\tau)-P_{m}W(v(\tau))|_{X}^{2}<\infty.
		\end{align}
		Since the integrand is converging, $\eqref{jj3}$, then the limit of the integrant must be zero, i.e.,
		\begin{align*}
			|\bar{v}-P_{m}W(\bar{v})|_{X}^{2}=0.
		\end{align*}
		Therefore $v(\tau)$ converges, as $\tau \rightarrow \infty$, to $\bar{v}$ which is the steady state solution of the determining form.
	  \item Follows from the facts established in item (3) above and Theorem \ref{DMW}.
	\end{enumerate}
\end{proof}

\section{Continous data assimilation}
In this section we provide a continuous data assimilation algorithm for recovering a reference solution $u$ of $\eqref{KdV}$-$\eqref{boundary}$ from $P_{m}u(s)$, for $s\geq s_{0}$. In other words, this is a downscaling algorithm for recovering the fine scales $(I-P_{m})u$ from knowing only the coarse scales $P_{m}u$.
\newtheorem{wgoesu}[ANA]{Theorem}
\begin{wgoesu}
\label{CDA}
Let $\rho=4R$ with $R$ given in $\eqref{Rkdv1}$, and let $u(s)$ be the solution of the damped and driven KdV equation $\eqref{KdV}$-$\eqref{boundary}$, with $f\in \dot{H}_{\text{per}}^2$. Assume that $|P_{m}u(s)|_{H^{2}}<\rho$, for all $s\geq s_{0}$. Let $w(s)$ be the solution of the corresponding `continuous data assimilation' equation
	\begin{align}
		w_s+ww_{x}+w_{xxx}+\gamma w=f- \mu [P_mw-P_{m}u],  \label{originalDA}
	\end{align}
with an arbitrary initial data $w(s_{0})\in \dot{H}^{2}$, with $|w(s_{0})|_{H^{2}}<\rho$, and subject to periodic boundary conditions
	\begin{align*}
		\partial_{x}^{j}w(s, x)=\partial_{x}^{j}w(s, x+L), \qquad \forall (s, x)\in \mathbb{R}\times \mathbb{R}. 	
	\end{align*}
for $j=0,1$ and $2$.
Assume that $\mu$ and $m$ are large enough so that conditions $\eqref{condition1}$, $\eqref{condition2}$, and
	\begin{align}
		C_{3}^{0}\leq \mu,  \label{condition3'} 
	\end{align}
	\begin{align}
		\frac{C_{3}^{0}L^{2}}{4\pi^{2}(m+1)^{2}}\leq \frac{\gamma}{2m}, \label{condition5} 
	\end{align}
	\begin{align}
		\frac{1}{m}\left((\gamma+2\mu)\mathcal{R}_{\infty}+2\mathcal{R^{'}}^{4}\gamma^{-3}\right)\leq \frac{\gamma}{2}, \label{condition6} 
	\end{align}
hold. Then we have $|w(s)- u(s)|\rightarrow 0$, as $s\rightarrow \infty$, at an exponential rate with exponent $\frac{\gamma}{4}$.
\end{wgoesu}

\begin{proof} We first note that one can prove the existence and the uniqueness of the solution for the initial value problem $\eqref{originalDA}$ subject to boundary conditions by slight adjustment of the proof given for KdV equation in \cite{Temam1}. In particular, one needs to regularize the equation with $\epsilon w_{xxxx}$ as in section 3, and look at the regularized equation with $C^{\infty}$ initial data $w^{\epsilon}(s_{0})$, where $w^{\epsilon}(s_{0})\rightarrow w(s_{0})$ in $H^{2}$, as $\epsilon \rightarrow 0$. Then one can use results in \cite{LionsPar}, and find uniform in $\epsilon$ estimates as in \cite{Temam1} (or as in section 4) to prove the desired result. Thus we can assume that both of the above equations $\eqref{KdV}$ and $\eqref{originalDA}$ have global solution in $L^{\infty}([s_{0}, \infty),\dot{H}^{2})$, and the equations hold in $L^{\infty}((s_{0}, \infty), \dot{H}^{-1})$. We also use the same bounding notation such as $\mathcal{R}_{0}$, $\mathcal{R}_{1}$, $\mathcal{R}_{2}$, $\mathcal{R}_{\infty}$ and $\mathcal{R}^{'}$ for $L^{2}$, $H^{1}$, $H^{2}$, $L^{\infty}$ and time derivative bounds for the global solutions of these equations where $s\geq s_{0}$, not on whole line $\mathbb{R}$. We note that since $|P_{m}u(s)|<\rho$, for all $s\geq s_{0}$, and $|w(s_{0})|_{H^{2}}<\rho$, the bounding expressions mentioned above depend on $\rho$. We proceed as in Theorem \ref{uniq1}, but take $P_{m}u$ instead of $v$ and $\delta=w-u$, to obtain
	\begin{align} \label{ekle1}
		\frac{d}{ds}\Psi+\gamma \Psi\leq \left[ (\gamma+2\mu)\mathcal{R}_{\infty}+2\mathcal{R^{'}}^{4}\gamma^{-3}\right]|\delta(s)|^{2}, 
 	\end{align}	
where again
	\begin{align}
		\Psi(\delta)=\int (\delta_{x}^{2}-\xi\delta^{2}). \label{F3'} 
 	\end{align}
As in the proof of Theorem \ref{uniq1}, we have
	\begin{align*}
		\frac{d}{ds}|\delta|^{2}+2\gamma |\delta|^{2}+2\mu |P\delta|^{2}=-\int 2(\xi \delta)_{x}\delta,
	\end{align*}
as well as in $\eqref{eqn66}$
	\begin{align*}
		-\int 2(\xi \delta)_{x}\delta \leq C_{3}^{0}|\delta|^{2},
	\end{align*}					
where 	
	\begin{align}
		C_{3}^{0}:=  (\mathcal{R}_{1}^{0})^{\frac{1}{2}}(\mathcal{R}_{2}^{0})^{\frac{1}{2}}. \label{c30}
	\end{align}
We note that, contrary to previous sections where $C_{3}$ depends on $\mu$ and $\rho$, here $C_{3}^{0}$ depends only on $\rho$, but not on $\mu$. Here we have, 	
	\begin{align*}						
		-\int 2(\xi \delta)_{x}\delta&\leq C_{3}^{0}|P\delta|^{2}+C_{3}^{0}|Q\delta|^{2}\leq C_{3}^{0}|P\delta|^{2}+ \frac{C_{3}^{0}L^{2}}{4\pi^{2}(m+1)^{2}}|\delta_{x}|^{2}.
         \end{align*}
If we choose $\mu$ large enough such that $\eqref{condition3'}$ hold, then
	\begin{align}
		\frac{d}{ds}|\delta|^{2}+2\gamma |\delta|^{2}\leq  \frac{C_{3}^{0}L^{2}}{4\pi^{2}(m+1)^{2}}|\delta_{x}|^{2}. \label{jj}
	\end{align}
Since
	\begin{align*}
		\Psi(\delta)\geq |\delta_{x}|^{2}-|\xi|_{\infty}|\delta|^{2}\geq |\delta_{x}|^{2}-\mathcal{R}_{\infty}|\delta|^{2},		
	\end{align*}
we have that
	\begin{align} \label{eqn6}
		|\delta_{x}|^{2}\leq \Psi(\delta)+\mathcal{R}_{\infty}|\delta|^{2}.
	\end{align}	
Thus,
	\begin{align}
		\frac{d}{ds}|\delta|^{2}+2\gamma |\delta|^{2}\leq  \frac{C_{3}L^{2}}{4\pi^{2}(m+1)^{2}}\Psi(\delta)+ \frac{C_{3}L^{2}\mathcal{R}_{\infty}}{4\pi^{2}(m+1)^{2}}|\delta|^{2}. \label{ekle2}
	\end{align}
We add $\frac{1}{m}\times\eqref{ekle1}$ and $\eqref{ekle2}$, to get
	\begin{align}
		\frac{d}{ds}\left(|\delta|^{2}+\frac{1}{m}\Psi\right)+\gamma \left(2|\delta|^{2}+\frac{1}{m}\Psi\right)\leq&  \frac{C_{3}L^{2}}{4\pi^{2}(m+1)^{2}}\Psi(\delta)+ \frac{C_{3}L^{2}\mathcal{R}_{\infty}}{4\pi^{2}(m+1)^{2}}|\delta|^{2}\notag\\
		&+\frac{1}{m}\left[ (\gamma+2\mu)\mathcal{R}_{\infty}+2\mathcal{R^{'}}^{4}\gamma^{-3}\right]|\delta|^{2}. \label{j2}
	\end{align}
We note that the functional $\Psi(\delta(\cdot)): [s_{0}, \infty) \rightarrow \mathbb{R}$, defined in $\eqref{F3'}$, is an absolutely continuous map with respect to $s$ on any subinterval of $[s_{0}, \infty)$. Now we look at different cases depending on the sign of $\Psi(\delta(s))$ for different values of $s$. \\\\
$\bf{CASE}$ $\bf{1}:$ $\Psi(\delta(s))\leq 0$ on the interval $[s^{*}, \infty)$. Then, by $\eqref{eqn6}$, we have
	\begin{align*}
		|\delta_{x}(s)|^{2}\leq \mathcal{R}_{\infty}|\delta(s)|^{2}, \text{ on the interval } [s^{*}, \infty).
	\end{align*}
From $\eqref{jj}$,
	\begin{align*}
		\frac{d}{ds}|\delta(s)|^{2}+2\gamma |\delta(s)|^{2}\leq  \frac{C_{3}L^{2} \mathcal{R}_{\infty}}{4\pi^{2}(m+1)^{2}}|\delta(s)|^{2}, \text{ on the interval } [s^{*}, \infty).
	\end{align*}
$\eqref{condition5}$ and $\eqref{condition6}$ imply that
	\begin{align*}
		\frac{C_{3}L^{2} \mathcal{R}_{\infty}}{4\pi^{2}(m+1)^{2}}\leq \gamma,
	\end{align*}
then by Gronwall's inequality,
	\begin{align*}
		|\delta(s)|^{2}\leq  e^{-\gamma(s-s^{*})}|\delta(s^{*})|^{2}, \text{ for every } s\in [s^{*}, \infty).
	\end{align*}
Thus $|\delta(s)|\rightarrow 0$, as $s\rightarrow \infty$, at an exponential rate with exponent $\frac{\gamma}{2}$.  \\\\
$\bf{CASE}$ $\bf{2}:$ $\Psi(\delta(s))\geq 0$ on the interval $[s^{*}, \infty)$. Then since $\Psi(\delta(s))\geq 0$, if we choose $m$ large enough such that $\eqref{condition5}$ and $\eqref{condition6}$ hold, then from $\eqref{j2}$, we obtain
	\begin{align*}
		\frac{d}{ds}\left(|\delta|^{2}+\frac{1}{m}\Psi\right)+\gamma \left(2|\delta|^{2}+\frac{1}{m}\Psi\right)\leq \frac{\gamma}{2m}\Psi+ \left(\frac{\gamma \mathcal{R}_{\infty}}{2m}+\gamma \right)|\delta|^{2}.
	\end{align*}
$\eqref{condition6}$ implies that $2\mathcal{R}_{\infty}\leq m$. Thus, we have
	\begin{align*}
		\frac{d}{ds}\left(|\delta|^{2}+\frac{1}{m}\Psi\right)+\frac{\gamma}{2} \left(|\delta|^{2}+\frac{1}{m}\Psi\right)\leq 0.
	\end{align*}
By Gronwall's inequality,
	\begin{align}
		|\delta(s)|^{2}+\frac{1}{m}\Psi(\delta(s))\leq \left(|\delta(s^{*})|^{2}+\frac{1}{m}\Psi(s^{*})\right)e^{-\frac{\gamma}{2}(s-s^{*})}, \text{ for all } s\in [s^{*}, \infty).\label{onemli1}
	\end{align}
Since $\Psi(\delta(s))\geq 0$ on the interval $[s^{*}, \infty)$, we have
	\begin{align*}
		|\delta(s)|^{2}\leq \left(|\delta(s^{*})|^{2}+\frac{1}{m}\Psi(s^{*})\right)e^{-\frac{\gamma}{2}(s-s^{*})}, \text{ for all } s\in [s^{*}, \infty).
	\end{align*}
We note that $|\delta(s^{*})|^{2}+\frac{1}{m}\Psi(s^{*})$ is bounded. Thus $|\delta(s)|\rightarrow 0$, as $s\rightarrow \infty$, at an exponential rate with exponent $\frac{\gamma}{4}$.\\\\
$\bf{CASE}$ $\bf{3}:$ Sign of $\Psi(\delta(s))$ alternates between subintervals on the interval $[s_{0}, \infty)$. Without loss of generality, we assume that there exists a sequence $\{s_{n}\}_{n\in \mathbb{N}}\subset [s_{0}, \infty)$ such that $\Psi(\delta(s))\geq 0$ on intervals $I_{2k+1}$ for $k\in {0, 1, 2, 3, ...}$, where $I_{2k+1}:=[s_{2k}, s_{2k+1}]$, $\Psi(\delta(s))\leq 0$ on intervals $I_{2k+2}$ for $k\in {0, 1, 2, 3, ...}$, where $I_{2k+2}:=[s_{2k+1}, s_{2k+2}]$, and $\Psi(\delta(s_{k}))=0$ for any nonnegative integer $k$. We note that $\cup_{n\in \mathbb{N}} I_{n}=[s_{0}, \infty)$. We prove the following lemma.

\newtheorem*{lemma}{Lemma}
	\begin{lemma}
	\label{lm2}
		$(a)$ For any nonzero $n\in \mathbb{N}$, $|\delta(s_{n})|^{2}\leq |\delta(s_{n-1})|^{2}e^{-\frac{\gamma}{2}(s_{n}-s_{n-1})}$. \\
		$(b)$ For any nonzero $n\in \mathbb{N}$, $|\delta(s_{n})|^{2}\leq |\delta(s_{0})|^{2}e^{-\frac{\gamma}{2}(s_{n}-s_{0})}$.\\
		$(c)$ For any $s\in (s_{0}, \infty)$, $|\delta(s)|^{2}\leq |\delta(s_{0})|^{2}e^{-\frac{\gamma}{2}(s-s_{0})}$.

	\end{lemma}
\begin{proof} [Proof of the Lemma]
$(a)$ When $n$ is an even integer, then $\Psi(\delta(s))\leq 0$ on the interval $I_{n}$. Thus, lemma follows from the first case above where we take the interval $I_{n}$ instead of $[s_{0}, \infty)$. When $n$ is an odd integer, then $\Psi(\delta(s))\geq 0$ on the interval $I_{n}$. Thus from the second case above, where we take  the interval $I_{n}$ instead of $[s_{0}, \infty)$, we obtain that
	\begin{align*}
		|\delta(s_{n})|^{2}\leq \left(|\delta(s_{n-1})|^{2}+\frac{1}{m}\Psi(s_{n-1})\right)e^{-\frac{\gamma}{2}(s_{n}-s_{n-1})}.
	\end{align*}
Since $\Psi(s_{n-1})=0$, Part $(a)$ follows. \\\\
$(b)$ This follows by mathematical induction and Part $(a)$. \\\\
$(c)$ For any $s\in (s_{0}, \infty)$, there exists an integer $n$ such that, $s\in I_{n+1}=[s_{n}, s_{n+1})$. By using the idea of the proof of Part (a), we obtain that
	\begin{align*}
		|\delta(s)|^{2}\leq |\delta(s_{n})|^{2}e^{-\frac{\gamma}{2}(s-s_{n})}.
	\end{align*}
Now by Part $(b)$,
	\begin{align*}
		|\delta(s)|^{2}&\leq |\delta(s_{0})|^{2}e^{-\frac{\gamma}{2}(s_{n}-s_{0})}e^{-\frac{\gamma}{2}(s-s_{n})}\\
				     &= |\delta(s_{0})|^{2}e^{-\frac{\gamma}{2}(s-s_{0})}.
	\end{align*}
This completes the proof of Lemma \ref{lm2}. 	
\end{proof}
Part $(c)$ shows that, $|\delta(s)|\rightarrow 0$, as $s\rightarrow \infty$, at an exponential rate with exponent $\frac{\gamma}{4}$.	
\end{proof}

\newtheorem{rky13}[rk1]{Remark}
	\begin{rky13}
		 The same technique where we combine two differential inequalities can be applied to give a data assimilation algorithm for damped and driven, nonlinear Schr\"odinger equations.			
	\end{rky13}

\section{Determining modes}
	\newtheorem{modeskdv}[ANA]{Theorem}
		\begin{modeskdv}
		\label{DM}
			Let $\rho=4R$, with $R$ is given in $\eqref{Rkdv1}$. Assume that $m$ is large enough such that
			\begin{align}			
			\frac{L^{2}}{4\pi^{2}(m+1)^{2}}\frac{1}{\gamma}\left[ 2\gamma \mathcal{R}_{\infty}^{0}+2(\mathcal{R^{'}}^{0})^{4}\gamma^{-3}\right]\leq \frac{1}{2}, \label{condition4'} 
			\end{align}
	which is the modified version of the condition $\eqref{condition4}$. Here $\mathcal{R}_{\infty}^{0}=\mathcal{R}_{\infty}|_{\mu=0}$ and $(\mathcal{R}^{'})^{0}=\mathcal{R}^{'}|_{\mu=0}$. Then the Fourier projection $P_m$ of $L^2$ onto the space $H_m$, where $H_m$ is defined in $\eqref{hm}$, is determining for $\eqref{KdV}$, i.e., if $u(\cdot), \bar{u}(\cdot)\subset \mathcal{A}$, the global attractor of $\eqref{KdV}$ with $P_mu(s)=P_m\bar{u}(s), \text{ } \text{for all} \text{ } s\in \mathbb{R}$, then $u(s)=\bar{u}(s),  \text{ } \text{for all} \text{ } s\in \mathbb{R}$.
		\end{modeskdv}
	\begin{proof}
		We assume $u(s)$ and $\bar{u}(s)$ are trajectories on the global attractor, $\mathcal{A}$, of $\eqref{KdV}$-$\eqref{boundary}$, and  $P_m(u(s))=P_m(\bar{u}(s))$ for all time $s\in \mathbb{R}$, where $m\in \mathbb{N}$ is such that $\eqref{condition4'}$ is satisfied. We have the following equations for $u$ and $\bar{u}$,
			$$u_s+uu_{x}+u_{xxx}+\gamma u=f,$$
			$$\bar{u}_s+\bar{u}\bar{u}_{x}+\bar{u}_{xxx}+\gamma \bar{u}=f.$$
		Subtract, denoting $\delta:=u-\bar{u}$, to obtain
		\begin{align*}
			\delta_s+(\xi\delta)_{x}+\delta_{xxx}+\gamma \delta= 0,
		\end{align*}
		where $\xi=\frac{u+\bar{u}}{2}$.	
We proceed as in the proof of the Theorem \ref{uniq1}, to get a `reverse' Poincar\'e type inequality
	\begin{align*}
		|\delta_{x}(s)|^{2}\leq \frac{1}{\gamma}\left[ 2\gamma(\mathcal{R}_{\infty})^{0}+2(\mathcal{R^{'}}^{0})^{4}\gamma^{-3}\right]\sup_{s\in \mathbb{R}}|\delta(s)|^{2}.
	\end{align*}
Since $P\delta=0$,
	\begin{align*}
		|\delta(s)|^{2}= |Q\delta(s)|^{2}&\leq \frac{L^{2}}{4\pi^{2}(m+1)^{2}}|\delta_{x}(s)|^{2}\\
						       &\leq \frac{L^{2}}{4\pi^{2}(m+1)^{2}}\frac{1}{\gamma}\left[ 2\gamma(\mathcal{R}_{\infty})^{0}+2(\mathcal{R^{'}}^{0})^{4}\gamma^{-3}\right]\sup_{s\in \mathbb{R}}|\delta(s)|^{2}\\
						       &\leq \frac{1}{2}\sup_{s\in \mathbb{R}}|\delta(s)|^{2}
	\end{align*}
due to the condition $\eqref{condition4'}$. Thus we obtain that $|\delta(s)|=0$ for all $s\in \mathbb{R}$, i.e.  $u(s)=\bar{u}(s)$, for all $s\in \mathbb{R}$. 
	\end{proof}

\newtheorem{rky1}[rk1]{Remark}
	\begin{rky1}
		 $\eqref{condition4'}$ implies that the number of determining modes $m=O(\gamma^{-\frac{26}{3}}, |f|_{H^{2}}^{\frac{14}{3}})$, as $\gamma \rightarrow 0$ and $|f|_{H^{2}}\rightarrow \infty$.\\
	\end{rky1}

\appendix

\section{Steady State solution}
We show the existence of a steady state solution of $\eqref{KdV}$-$\eqref{boundary}$, namely a solution of
	\begin{align}
		uu_{x}+u_{xxx}+\gamma u=f, \label{KdVSS}
	\end{align}
subject to periodic boundary conditions with basic periodic interval $[0, L]$. Here $\int f=0$, thus the solution also must satisfy the compatibility condition $\int u=0$. If $f=0$, then $u=0$ is a solution. Take any nonzero $f$. First, we show that $\eqref{KdVSS}$ has a weak solution, in the sense of distributions. To this end, we consider the Galerkin approximation of $\eqref{KdVSS}$ in $H_{n}\cong \mathbb{R}^{2n}$. The Galerkin approximation system is
	\begin{align}
		P_{n}(u_{n}(u_{n})_{x})+(u_{n})_{xxx}+\gamma u_{n}=P_{n}f, \label{gal}
	\end{align}
where $u_{n}\in H_{n}$. Our first goal is to show that $\eqref{gal}$ has a solution. To begin, we establish some a-priori estimates for $\eqref{gal}$. Suppose $u_{n}$ is a solution of $\eqref{gal}$. Take the inner product of $\eqref{gal}$ with $u_{n}$, and integrate by parts to obtain
	\begin{align*}
		\gamma |u_{n}|^{2}= (f, u_{n})\leq |f||u_{n}|\leq \frac{|f|^{2}}{2\gamma}+\frac{\gamma |u_{n}|^{2}}{2}.
	\end{align*}
Thus,
	\begin{align*}
		\frac{\gamma}{2} |u_{n}|^{2}\leq \frac{|f|^{2}}{2\gamma}.
	\end{align*}
From here, we conclude the following uniform in $n$ estimate for all solutions of $\eqref{gal}$
	\begin{align*}
		|u_{n}|\leq \frac{|f|}{\gamma}:=R_{0},
	\end{align*}
provided such solutions exist. We will use the following proposition that is proven in \cite{ConFoias}. 	
	\newtheorem*{CF*}{Propostion}
		\begin{CF*}
			Let $B(0, R)\subset \mathbb{R}^{n}$ be a ball of radius $R$. Suppose $\Phi: \bar{B}(0, R)\rightarrow \mathbb{R}^{n}$ is a continuous map such that $$(\Phi(v),v) <0 \text{ for every } v\in \partial B(0, R).$$
			Then there exists $v^{*}\in \bar{B}(0, R)$ such that
			$$\Phi(v^{*})=0.$$
		\end{CF*}	
We apply the above proposition to show that $\eqref{gal}$ has a solution. Let $\Phi$ be the map defined on a closed ball in $H_{n}$, with $L^{2}$ inner product, with radius $2R_{0}$, i.e, $\Phi: \bar{B}(0, 2R_{0})\rightarrow H_{n}$ where
		\begin{align*}
			\Phi(w)=-P_{n}(ww_{x})-w_{xxx}-\gamma w+P_{n}f.
		\end{align*}	
$\Phi$ is a continuous map since it is essentially a quadratic map. Let $|w|=2R_{0}$. Then
   	\begin{align*}
		(\Phi(w),w)&=-\gamma |w|^{2}+(f,w)\leq -\gamma |w|^{2}+|f||w| \\
				 &\leq |w|( -\gamma |w|+|f|)< 0.
	\end{align*}
Thus applying the above proposition for $R=2R_{0}$, we conclude that there exists $w^{*}\in B(0, 2R_{0})$ such that $\Phi(w^{*})=0$; $\eqref{gal}$ has a solution $w^{*}=: u_{n}$. Now, we establish additional stronger norms on the solution of $u_{n}$ of $\eqref{gal}$. Take the $L^{2}$ inner product of $\eqref{gal}$ with $(u_{n})_{xxx}\in H_{n}$ and interpolate to get
   	\begin{align*}
		|(u_{n})_{xxx}|^{2}&= -\int u_{n}(u_{n})_{x}(u_{n})_{xxx} -\gamma \int u_{n}(u_{n})_{xxx}+\int f(u_{n})_{xxx}\\
					     &\leq |u_{n}|_{\infty}|(u_{n})_{x}||(u_{n})_{xxx}|+ |f| |(u_{n})_{xxx}|\\
					     &\leq |u_{n}|^{\frac{1}{2}}|(u_{n})_{x}|^{\frac{1}{2}}|(u_{n})_{x}||(u_{n})_{xxx}|+ |f||(u_{n})_{xxx}|\\
					     &\leq c|u_{n}|^{\frac{3}{2}}|(u_{n})_{xxx}|^{\frac{3}{2}}+|f||(u_{n})_{xxx}|\\
					     &\leq cR_{0}^{\frac{3}{2}}|(u_{n})_{xxx}|^{\frac{3}{2}}+|f||(u_{n})_{xxx}|\\
					     &\leq \frac{cR_{0}^{6}}{4}+ \frac{3}{4}|(u_{n})_{xxx}|^{2}+2|f|^{2}+\frac{1}{8}|(u_{n})_{xxx}|^{2}\\
					     &=\frac{cR_{0}^{6}}{4}+2|f|^{2}+ \frac{7}{8}|(u_{n})_{xxx}|^{2}.
	\end{align*}
Thus, $|(u_{n})_{xxx}|^{2}\leq 2cR_{0}^{6}+16|f|^{2}:=R_{3}^{2}$. Hence $u_{n}$ satisfies the uniform in $n$ bounds
   	\begin{align*}
		|u_{n}|\leq R_{0}, \quad |(u_{n})_{xxx}|\leq R_{3}.
	\end{align*}
By the Sobolev embedding theorem, there exists a subsequence $u_{n_{j}}$ such that
   	\begin{align*}
		&u_{n_{j}} \rightarrow u^{*} \text{ weakly in } H^{3},\\
		&u_{n_{j}} \rightarrow u^{*} \text{ strongly in } H^{\sigma},
	\end{align*}
for all $\sigma\in [0, 3)$, in particular for $\sigma=2$, and $u^{*}\in H^{3}$ satisfies
   	\begin{align*}
		|u^{*}|\leq R_{0}, \quad |u^{*}_{xxx}|\leq R_{3}.
	\end{align*}
Let $\mathcal{V}:= \{ \text{all trigonometric polynomials with average zero} \}$ be a set of test functions. Let $\phi\in \mathcal{V}$, and take $n_{j}>\text{degree}\{\phi\}$. Then from $\eqref{gal}$, we have
	\begin{align*}
		- \frac{1}{2}(u_{n_{j}}^{2}, \partial_{x} \phi)- (u_{n_{j}}, \partial_{x}^{3} \phi)+\gamma (u_{n_{j}}, \phi)= (f, \phi).
	\end{align*}
Since $u_{n_{j}}\rightarrow u^{*}$ strongly in $\dot{L}^{2}$, we have
	\begin{align*}
		&\lim_{j\rightarrow \infty} - (u_{n_{j}}, \partial_{x}^{3} \phi)= - (u^{*}, \partial_{x}^{3} \phi), \\
		&\lim_{j\rightarrow \infty} \gamma (u_{n_{j}}, \phi)= \gamma (u^{*}, \phi).
	\end{align*}	
Now,
	\begin{align*}
		|(u_{n_{j}}^{2}-(u^{*})^{2}, \partial_{x}\phi)|&= |\int (u_{n_{j}}-u^{*})(u_{n_{j}}+u^{*})\partial_{x}\phi| \\
									 &\leq |u_{n_{j}}-u^{*}||u_{n_{j}}+u^{*}||\partial_{x}\phi|_{\infty}\\
									 &\leq 2R_{0}|u_{n_{j}}-u^{*}||\partial_{x}\phi|_{\infty}\rightarrow 0,
	\end{align*}
as $j\rightarrow \infty$, since $u_{n_{j}}\rightarrow u^{*}$ strongly in $\dot{L}^{2}$. Thus, $\lim_{j\rightarrow \infty} - \frac{1}{2}(u_{n_{j}}^{2}, \partial_{x} \phi)=- \frac{1}{2}((u^{*})^{2}, \partial_{x} \phi)$. Thus we have shown that $u^{*}\in \dot{H}^{3}$ satisfying
	\begin{align*}
		|u^{*}|\leq R_{0}, \quad |u_{xxx}^{*}|\leq R_{3},
	\end{align*}
and that $\eqref{KdVSS}$ is satisfied in the sense of distributions,
	\begin{align*}
		- \frac{1}{2}((u^{*})^{2}, \partial_{x} \phi)- (u^{*}, \partial_{x}^{3} \phi)+\gamma (u^{*}, \phi)= (f, \phi),
	\end{align*}
for every $\phi\in \mathcal{V}$. Since $u^{*}\in \dot{H}^{3}$, and since $H^{3}$ is an algebra, we have $(u^{*})^{2}\in H^{3}$ and $u^{*}u^{*}_{x}\in H^{2}$. Therefore, the equation
	\begin{align*}
 		u^{*}(u^{*})_{x}+(u^{*})_{xxx}+\gamma u^{*}=f
	\end{align*}
holds in $L^{2}$. Then,
	\begin{align*}
		(u^{*})_{xxx}=-u^{*}(u^{*})_{x}-\gamma u^{*}+f.
	\end{align*}
Since $f\in H^{2}$, and $u^{*}(u^{*})_{x}\in H^{2}$, we have $(u^{*})_{xxx}\in H^{2}$. This implies that $u^{*}\in H^{5}$. Thus
	\begin{align*}
		u^{*}(u^{*})_{x}+(u^{*})_{xxx}+\gamma u^{*}=f \text{ holds in } H^{2}.
	\end{align*}
In particular, since $H^{2}\subset C^{1}$, the equation holds in the classical sense and satisfies the estimates above.  		

\section*{Acknowledgements}
The work of M.S.J. and T.S. have been partially supported by National Science Foundation (NSF) Grant Nos.DMS-1418911. The work of E.S.T. was  supported in part by the ONR grant N00014-15-1-2333 and the NSF grants DMS-1109640 and DMS-1109645.

\end{document}